\documentclass[11pt,a4paper]{amsart}
\usepackage[utf8]{inputenc}
\pdfoutput=1

\usepackage[english]{babel}
\usepackage[top=3cm,bottom=3cm,left=3cm,right=3cm]{geometry}
\usepackage{amsmath}
\usepackage{amssymb}
\usepackage{amsfonts} 
\usepackage{amsthm}
\usepackage{mathrsfs}
\usepackage{graphicx}
\usepackage{enumitem}
\setlist[itemize]{leftmargin=*}
\usepackage{xcolor}
\usepackage{url}
\usepackage{accents}
\usepackage{bm}
\usepackage[all]{xy}
\usepackage{xcolor}
\usepackage{hyperref}
\hypersetup{
	colorlinks,
	linkcolor={blue!60!black},
	citecolor={green!60!black},
	urlcolor={green!60!black}
}
\usepackage{doi}
\usepackage[capitalise,nameinlink]{cleveref}
\crefformat{equation}{#2(#1)#3}
\crefrangeformat{equation}{#3(#1)#4 to~#5(#2)#6}
\crefmultiformat{equation}{#2(#1)#3}{ and~#2(#1)#3}{, #2(#1)#3}{ and~#2(#1)#3}
\usepackage{caption}
\usepackage{subcaption}
\usepackage{booktabs}
\usepackage{mathtools}
\usepackage{etoolbox}
\usepackage{overpic}
\usepackage{xspace}
\usepackage{comment}
\usepackage{chngcntr}
\usepackage{bold-extra}
\usepackage[stretch=30,shrink=30]{microtype}
\newcommand{\N}{\mathbb{N}}
\newcommand{\Z}{\mathbb{Z}}

\newcommand{\R}{\mathbb{R}}
\newcommand{\C}{\mathbb{C}}

\newcommand\restr[2]{{
		\left.\kern-\nulldelimiterspace
		#1
		\vphantom{\big|}
		\right|_{#2}
}}
\newcommand{\de}{\partial}

\newcommand{\mz}{\frac{1}{2}}

\newcommand{\ang}[1]{\left\langle#1\right\rangle}
\newcommand{\Bang}[1]{\big\langle#1\big\rangle}
\newcommand{\uno}{\bm{1}}

\newcommand{\nin}{\not\in}

\newcommand{\weakto}{\rightharpoonup}

\newcommand{\mrestr}{\mathbin{\vrule height 1.6ex depth 0pt width
		0.13ex\vrule height 0.13ex depth 0pt width 1.3ex}}
\newcommand{\cptsub}{\subset\hspace{-1pt}\subset}
\renewcommand{\bar}{\overline}

\theoremstyle{definition}
\newtheorem{definition}{Definition}[section]

\newtheorem{rmk}[definition]{Remark}

\newtheorem*{definition*}{Definition}
\newtheorem*{notazen*}{Notation}
\newtheorem*{rmk*}{Remark}
\newtheorem*{example*}{Example}
\newtheorem*{ack*}{Acknowledgement}
\newtheorem*{acks*}{Acknowledgements}

\theoremstyle{plain}
\newtheorem{thm}[definition]{Theorem}
\newtheorem{lemmaen}[definition]{Lemma}
\newtheorem{corollary}[definition]{Corollary}
\newtheorem{proposition}[definition]{Proposition}

\newtheorem*{thm*}{Theorem}
\newtheorem*{lemmaen*}{Lemma}
\newtheorem*{corollary*}{Corollary}
\newtheorem*{proposition*}{Proposition}
\newtheorem*{claim*}{Claim}
\newtheorem*{conj*}{Conjecture}

\counterwithin{equation}{section}

\newcommand{\vfd}{\mathbf{v}}
\newcommand{\envdim}{Q}
\newcommand{\subman}{\mathcal{M}}
\newcommand{\bsubman}{\mathcal{N}}
\newcommand{\bman}{\mathfrak{M}}
\newcommand{\fbvf}{\mathcal{X}_{fb}}

\newcommand{\ff}{\mathrm {I\!I}}
\renewcommand{\hat}{\widehat}
\renewcommand{\tilde}{\widetilde}
\renewcommand{\epsilon}{\varepsilon}
\DeclareMathOperator{\riem}{Rm}
\DeclareMathOperator{\spt}{spt}

\begin{document}
	\title[The viscosity method for min-max free boundary minimal surfaces]{The viscosity method for min-max \\ free boundary minimal surfaces}
	\author{Alessandro Pigati}
	\address{ETH Z\"urich, Department of Mathematics,
	R\"amistrasse 101, 8092 Z\"urich}
	\email{alessandro.pigati@math.ethz.ch}
	
	\begin{abstract}
		We adapt the viscosity method introduced by Rivière in \cite{riv.minmax} to the free boundary case.
		Namely, given a compact oriented surface $\Sigma$, possibly with boundary, a closed ambient Riemannian manifold $(\subman^m,g)$ and a closed embedded submanifold $\bsubman^n\subset\subman$, we study the asymptotic behavior of (almost) critical maps $\Phi$ for the functional
		\begin{align*}
			&E_\sigma(\Phi):=\operatorname{area}(\Phi)+\sigma\operatorname{length}(\Phi|_{\de\Sigma})+\sigma^4\int_\Sigma|\ff^\Phi|^4\,\operatorname{vol}_\Phi
		\end{align*}
		on immersions $\Phi:\Sigma\to\subman$ with the constraint $\Phi(\de\Sigma)\subseteq\bsubman$, as $\sigma\to 0$,
		assuming an upper bound for the area and a suitable entropy condition.
		
		As a consequence, given any collection $\mathcal{F}$ of compact subsets of the space of smooth immersions $(\Sigma,\de\Sigma)\to(\subman,\bsubman)$, assuming $\mathcal{F}$ to be stable under isotopies of this space we show that the min-max value
		\begin{align*}
			&\beta:=\inf_{A\in\mathcal{F}}\max_{\Phi\in A}\operatorname{area}(\Phi)
		\end{align*}
		is the sum of the areas of finitely many branched minimal immersions
		$\Phi_{(i)}:\Sigma_{(i)}\to\subman$ with $\de_\nu\Phi_{(i)}\perp T\bsubman$ along $\de\Sigma_{(i)}$, whose (connected) domains $\Sigma_{(i)}$ can be different from $\Sigma$ but cannot have a more complicated topology.
		
		We adopt a point of view which exploits extensively the diffeomorphism invariance of $E_\sigma$ and, along the way, we simplify several arguments from the original work \cite{riv.minmax}.
		Some parts generalize to closed higher-dimensional domains, for which we get a rectifiable stationary varifold in the limit.
	\end{abstract}
	
	\maketitle

	\section{Introduction}\label{intro.sec}
	\subsection{Min-max theories for minimal submanifolds}
	
	The study of minimal surfaces, namely surfaces which are critical for the area functional,
	has always been a central topic in geometric analysis and stimulated huge developments in the calculus of variations, geometric measure theory, partial differential equations and differential geometry.
	
	Among the most important questions in the calculus of variations is Plateau's problem,
	which was actually posed by Lagrange in 1760, asking to find a surface of least area for any assigned (smooth, closed) one-dimensional boundary $\Gamma\subset\R^3$.
	This question was famously resolved, independently, by Douglas and Rad\'o in 1930--1931 \cite{douglas,rado}, for a connected $\Gamma$ and with $\Sigma$ minimizing the area among (branched) immersed disks.
	The method is based on the fact that any immersion of the disk can be reparametrized to be conformal, so that the area becomes the Dirichlet energy; this functional has much better analytic properties and admits only a finite dimensional invariance group, whereas the area is invariant for the whole group of diffeomorphisms of the domain.
	
	In order to construct higher dimensional minimal submanifolds, an approach involving parametrizations does not seem to be available.
	On the other hand, bringing together the theories of rectifiable sets and De Rham's currents, Federer and Fleming \cite{fedfle} developed the modern \emph{theory of currents}, which allows to have an intrinsic weak notion of submanifold,
	together with notions of boundary and area (the \emph{mass}). Most importantly, there exists a natural topology which makes \emph{integral currents} a compact space---for instance with assigned boundary and an upper bound for the mass---and the mass a lower semicontinuous functional.
	These properties make it straightforward to apply the direct method of calculus of variations, in order to find a mass-minimizing object $\Sigma^k$ with an assigned boundary $\de\Sigma^k=\Gamma^{k-1}$ in $\R^m$, provided $\Gamma$ is itself a cycle and $1<k\le m$.
	
	The higher dimensional version of Plateau's problem thus reduced to the problem of reaching a satisfactory regularity theory for the minimizer $\Sigma$.
	This task was accomplished in codimension $1$, namely $k=m-1$, away from the boundary $\Gamma$---a crowning achievement which came from the work of De Giorgi, Fleming, Almgren, Simons and Federer.
	The theory of currents is flexible enough to produce mass-minimizing cycles in a given homology class of a closed Riemannian manifold.
	
	The obvious next question is to be able to produce general critical points for the area, either in $\R^m$ with a boundary constraint or in a closed curved ambient. The following conjecture by Yau, whose statement echoes the same question for geodesics, attracted a lot of attention in the last decades.
	
	\begin{conj*}[Yau, 1982]
		Does every closed three-dimensional Riemannian manifold $(\subman^3,g)$ admit infinitely many closed immersed minimal surfaces?
	\end{conj*}
	
	A very robust method to produce general, possibly unstable critical points is by means of min-max problems, and goes back to the work of Birkhoff on the existence of closed geodesics on the 2-sphere with an arbitrary metric: see \cite{malchiodi} for an introduction and a large collection of examples implementing this idea.
	
	The main issue is then how to implement a min-max construction in the setting of minimal surfaces.
	A successful theory was proposed by Almgren and Pitts \cite{almgren,pitts}: within the theory of currents, using cycles mod $2$ they produce an \emph{almost minimizing varifold} in the limit. The notion of \emph{varifold}, for which the reader may consult \cite[Chapters~4~and~8]{simon},
	differs from the one of current in that, while also retaining good compactness properties, the mass becomes continuous under weak convergence: this property is essential to guarantee that the limit object attains the min-max value---on the other hand, lower semicontinuity of the mass for currents is just good enough for minimization problems.
	
	While the regularity theory for general integer rectifiable stationary varifolds is wide open, with notable exceptions provided by the small-excess (and multiplicity-one) regularity theorem by Allard \cite{allard} and the deep structure result in the stable, codimension one case by Wickramasekera \cite{wick}, the technical requirement of \emph{almost minimality} enabled Pitts---together with later work by Schoen and Simon \cite{schoen}---to recover the full regularity for the limit varifold, proving that the singular set is empty in ambient dimension $3\le m\le 7$.
	
	The use of this min-max framework led to the solution of several long-standing problems, including the Willmore conjecture, by Marques and Neves \cite{mn.willmore}, and the Yau conjecture, by Marques, Neves, Song and others \cite{mn.ric,mn.irie,mn.equi,song}.
	
	This theory was used also to construct \emph{free boundary} minimal hypersurfaces:
	given an ambient $\subman^m$ and a submanifold $\bsubman$---usually $\subman$ has no boundary, or $\bsubman$ is precisely $\de\subman$---they are hypersurfaces $\Sigma^{m-1}$ with boundary, embedded or immersed in $\subman$, which are critical for the $(m-1)$-area with the constraint $\de\Sigma\subseteq\bsubman$.
	This is equivalent to the fact that $\Sigma$ is minimal and meets $\bsubman$ orthogonally along $\de\Sigma$.
	
	The most studied case is $(\subman,\bsubman)=(\bar B^3,S^2)$ with the Euclidean metric. In \cite{ketover}, using an equivariant version of the Simon--Smith theory---which is itself a less technical and more effective relative of Almgren--Pitts for surfaces, allowing for instance to control the genus of the resulting min-max surface---Ketover built free boundary minimal surfaces in the ball with arbitrarily big genus and three boundary components. In the same spirit, a very recent work by Carlotto, Franz and Schulz \cite{cfs} constructs surfaces with connected boundary and arbitrary genus.
	
	In \cite{lz1,lz2} the Almgren--Pitts theory for hypersurfaces in arbitrary dimension is adapted to the free boundary case.
	We also mention \cite{delellis} for a similar min-max theory in the free boundary case, avoiding discretized constructions.
	
	Several other techniques are used to construct free boundary minimal submanifolds, including notably desingularization methods and the study of extremal eigenvalue problems; for a survey of recent results, we invite the reader to consult \cite{li.survey}.
	
	Recently, in the closed case, another approach using the Allen--Cahn functional was proposed by Guaraco \cite{guaraco}.
	This theory, which started with the work of Modica \cite{modica} for minimizers and Hutchinson--Tonegawa \cite{hut.ton} for general critical points, interprets a minimal hypersurface as a limit interface of a phase transition, hence as a sort of limit level set of functions which are critical for rescalings of the Allen--Cahn energy, which should be seen as a relaxation of the area for the level sets.
	This approach seems to be at least as powerful as Almgren--Pitts; the additional structure given by having a sequence of smooth \emph{critical} functions converging to the limit already allowed to obtain finer results: see, e.g., the recent works by Chodosh--Mantoulidis \cite{ch.ma} and Bellettini \cite{bell}.
	
	In codimension two, interesting attempts have been made using the Ginzburg--Landau energy for complex valued maps, by Cheng \cite{cheng} and Stern \cite{stern}. This functional, which appears formally identical to Allen--Cahn---the latter being just Ginzburg--Landau for real valued maps---exhibits a totally different behavior in terms of energy concentration, due to the dominance of the angular part of the map in the Dirichlet term.
	This component forces the asymptotic analysis to take place on infinitely many scales, making the study very challenging.
	A different attempt, based on rescalings of the Yang--Mills--Higgs energy for sections and connections of a Hermitian line bundle,
	was proposed by the author and Stern \cite{pig.ste}. In this last framework, the asymptotic analysis becomes much simpler and quite similar to the Allen--Cahn setting, although a regularity theory still lacks.
	
	Yet another framework, which is the one whose study is continued in this paper, was introduced by Rivière \cite{riv.minmax}.
	It concerns minimal surfaces, but works in arbitrary codimension. As in the classical works \cite{douglas,sa.uh}, it uses parametrizations $\Phi:\Sigma^2\to(\subman^m,g)$. On the other hand, the area is not immediately relaxed with the Dirichlet energy; rather, it uses the functional
	\begin{align*}
		&E_\sigma'(\Phi):=\operatorname{area}(\Phi)+\sigma^2\int_\Sigma(1+|\ff^\Phi|^2)^p\,\operatorname{vol}_\Phi
	\end{align*}
	for $\sigma>0$ and a fixed exponent $p>1$, where the norm of the second fundamental form $\ff^\Phi$ and the area element $\operatorname{vol}_\Phi$ are with respect to the metric $\Phi^*g$ induced by $\Phi$. By studying critical points for $E_\sigma'$, one hopes to get a limit minimal immersion regardless of the topology of the closed surface $\Sigma$, while in the work by Sacks--Uhlenbeck \cite{sa.uh}---which relaxes the Dirichlet energy---one can just reach a harmonic map, whose minimality is not guaranteed unless $\Sigma$ is a sphere.
	As for the free boundary case, minimality holds automatically only if $\Sigma$ is a disk---a fact already exploited to solve Plateau's problem;
	in fact, we mention that the same approach developed in \cite{sa.uh} was used to build free boundary minimal disks in \cite{struwe}.
	Note that $E_\sigma'$ is invariant under diffeomorphisms, whereas the Dirichlet energy is only conformally invariant.
	
	The main outcome of \cite{riv.minmax} is that, once a sequence of maps $\Phi_k$ critical for $E_{\sigma_k}$ is carefully chosen,
	the induced varifolds converge to a kind of limit object called \emph{parametrized stationary varifold}.
	This notion defines a special class of varifolds which are induced by a parametrization $\Phi$---possibly with a new domain $\tilde\Sigma$ due to bubbling phenomena---and a Borel multiplicity function $N$ defined on $\tilde\Sigma$. A crucial feature is that stationarity can be localized with respect to the domain. This allows to obtain the full regularity for the limit object: the regularity theory was started in \cite{riv.target}, where $N\equiv 1$ is assumed, and carried out in full generality in \cite{pigriv}, where parametrized stationary varifolds are axiomatically studied.
	
	Later, in \cite{multone}, the authors show that actually $N\equiv 1$ in the variational setting, by exploiting the results from \cite{riv.minmax,pigriv}.
	This fact allows to obtain an upper bound on the Morse index of the limit minimal immersion in terms of the number of min-max parameters.
	In this sense it is the correct analogue of the \emph{multiplicity one conjecture} by Marques--Neves \cite{mn.multone},
	which has been confirmed very recently in the generic, codimension one case for the Almgren--Pitts framework \cite{zhou}.
	
	\subsection{Main results}
	In this paper we study a similar energy for surfaces with boundary; namely, choosing $p=2$ and replacing $\sigma$ with $\sigma^2$ for conveniency, we work with the energies
	\begin{align*}
		&E_\sigma(\Phi):=\operatorname{area}(\Phi)+\sigma\operatorname{length}(\Phi|_{\de\Sigma})+\sigma^4\int_\Sigma|\ff^\Phi|^4\,\operatorname{vol}_\Phi,
	\end{align*}
	where $\Sigma$ is a fixed compact oriented surface with (possibly nonempty) boundary, and $\Phi:\Sigma\to\subman^m$ is a smooth immersion with the constraint $\Phi(\de\Sigma)\subseteq\bsubman$.
	The parameter $\sigma$ should be thought dimensionally as a length.
	
	In this work we fully exploit the invariance of $E_\sigma$ under diffeomorphisms of the domain, namely the principle that every diffeomorphism invariant quantity should depend only on the shape of the immersed surface.
	In computing the first variation we will see that, using infinitesimal variations of the form $w=X(\Phi)$,
	all second-order terms involving $w$ are expressible just in terms of the second fundamental form of $\Phi$, as expected.
	A natural consequence of this is that the first variation of the relaxing terms $\sigma\operatorname{length}(\Phi|_{\de\Sigma})$
	and $\sigma^4\int_\Sigma|\ff^\Phi|^4\,\operatorname{vol}_\Phi$, for such special ambient deformations, can be bounded in terms of these quantities themselves (and the ambient vector field $X$).
	
	Also, working on a Finsler manifold $\bman$ of $W^{2,4}$ immersions, equipped with a norm on $T_\Phi\bman$ involving the induced metric $g_\Phi:=\Phi^*g$, we observe that also $\|X(\Phi)\|_\Phi$ is bounded in terms of $E_\sigma(\Phi)$, $X$ and $\sigma$.
	Since in the asymptotic analysis we will use only this particular kind of variations, we do not need to construct critical points of $E_\sigma$: it suffices to have $\|dE_\sigma(\Phi)\|_\Phi$ very small in terms of $\sigma$.
	Since such almost critical maps are easy to construct using pseudo-gradient flows and can be assumed, without loss of generality, to be smooth, this makes the paper self-contained---except for the regularity theory in \cref{reg.sec}.
	
	These observations, detailed in Sections \ref{crit.sec} and \ref{firstvar.sec}, represent a major simplification over the original work \cite{riv.minmax}, which appeals to \cite{riv.ps} for the Palais--Smale property of $E_\sigma'$ and the regularity of critical maps.
	The formulas obtained in this paper are quite simple, independently of the ambient: differently from \cite{riv.minmax}---where $\subman$ is assumed to be the round sphere $S^3$ in order to simplify the presentation---we can deal immediately with general closed manifolds $\subman$ and $\bsubman$.
	
	As in the closed case, the main difficulty is to prove a lower bound for the area of the immersed surface $\Phi$ in suitable balls $B_r(p)$ in the ambient. This is accomplished by studying how the ratio $\frac{\mu(B_s(p))}{s^2}$ behaves as $s$ varies, with $\mu$ denoting the area measure of $\Phi$ on $\subman$. While for $s<\sigma$ the boundedness of the quantity $\sigma^4\int_\Sigma|\ff^\Phi|^4\,\operatorname{vol}_\Phi$ is enough---in that, heuristically, magnifying by a factor $s^{-1}$ we get an $L^4$-bound on the second fundamental form and we can apply directly the monotonicity formula---for $s>\sigma$ we have to use the almost criticality of $\Phi$.
	
	Namely, we use the same vector fields used to show the (approximate) monotonicity of $\frac{\mu(B_s(p))}{s^2}$ for free boundary minimal surfaces, in order to understand the growth rate of this ratio for our immersed surface.
	Oversimplifying, the quantity $\frac{\sigma^4}{s}\int_\Sigma|\ff^\Phi|^4\,\operatorname{vol}_\Phi$ appears among the error terms: since this has to be integrated between $\sigma$ and $r$, this produces an error $\sigma^4\log(\sigma^{-1})\int_\Sigma|\ff^\Phi|^4\,\operatorname{vol}_\Phi$.
	As in \cite{riv.minmax}, this can be assumed to be infinitesimal with a careful selection of $\sigma$ and $\Phi$, based on Struwe's monotonicity trick for relaxed energies. In reality, the argument also requires a \emph{maximal} bound
	\begin{align*}
		&\sigma^4\int_{\Phi^{-1}(B_s(p))}|\ff^\Phi|^4\,\operatorname{vol}_\Phi\le\delta\mu(B_s(p))\quad\text{for all }s>0.
	\end{align*}
	
	We add the additional term $\sigma\operatorname{length}(\Phi|_{\de\Sigma})$ in $E_\sigma$ in order to deal with the additional challenge of having a nontrivial boundary $\Phi|_{\de\Sigma}$. Due to this, we cannot use the monotonicity formula on a ball $B_s(p)$ (with $s<\sigma$) whose preimage intersects $\de\Sigma$. In principle, one can impose a strong control of the boundary by adding a term involving the geodesic curvature of $\Phi|_{\de\Sigma}$; however, this would still require to understand the topology of $\Phi^{-1}(B_s(p))$.
	
	Rather, using a covering argument, we show that the set of points with distance less than $\sigma$ from $\Phi(\de\Sigma)$ has an area (i.e., the measure $\mu$) controlled by $\sigma\operatorname{length}(\Phi|_{\de\Sigma})$; this term is again infinitesimal as $\sigma\to 0$, so that this set can be ignored in the asymptotic analysis.
	
	The rest of the paper adapts the remaining arguments from \cite{riv.minmax} and \cite{riv.target} to the free boundary case---again with some important simplifications.
	In \cref{deg.sec} we study carefully what happens when the conformal structure induced by $\Phi$ degenerates as $\sigma\to 0$, which is more delicate and less well known for surfaces with boundary.
	
	The following is the main result of this work.
	
	\begin{thm}\label{main.thm}
		Let $(\subman^m,g)$ be a closed Riemannian manifold, $\bsubman^n\subset\subman$ a closed embedded submanifold (with $1\le n<m$),
		and let $\Sigma$ be a compact oriented surface, possibly with boundary.
		Given a sequence $\Phi_k$ of immersions which are $\sigma_k^5$-critical for $E_{\sigma_k}$, have bounded area and satisfy the condition
		\begin{align*}
			&\sigma_k\log\sigma_k^{-1}\operatorname{length}(\Phi_k|_{\de\Sigma})
			+\sigma_k^4\log\sigma_k^{-1}\int_\Sigma|\ff^{\Phi_k}|^4\,\operatorname{vol}_{\Phi_k}
			\to 0,
		\end{align*}
		there exists a subsequence such that the induced varifolds converge to a \emph{parametrized free boundary stationary varifold} for the couple $(\subman,\bsubman)$. Moreover, the connected components $\Sigma_i$ of its domain have
		$\chi(\Sigma_{i})\ge \chi(\Sigma)$ and $g(\Sigma_{i})\le g(\Sigma)$.
	\end{thm}
	
	In this statement $\chi(\cdot)$ is the Euler characteristic and $g(\cdot)$ is the genus. The last part of the statement follows from the analysis carried out in \cref{deg.sec}.
	We refer to \cref{par.def} for the precise description of this notion of parametrized varifold; the fact that one can localize the stationarity with respect to the domain stems from the fact that one can use variations $w=X(\Phi)$ also just on a domain $\omega\subset\Sigma$, extending $w$ to vanish on the complement, provided $X$ is supported far from $\Phi(\de\omega)$.
	
	As in the closed case, we can actually assert that the multiplicity function in the parametrized varifold is everywhere equal to $1$: see \cref{multone.rmk}.
		
	\begin{rmk}
		This result applies also to a compact ambient manifold $\subman$ with boundary $\bsubman$, such as the flat unit ball $\bar B^3$; note that the (almost) criticality should be understood formally, for infinitesimal variations $w$ which are sections of $\Phi^*T\subman$, with $w(\de\Sigma)\subseteq T\bsubman$.
		Indeed, we can smoothly extend $\subman$ to a closed Riemannian manifold and reduce to the previous statement.
	\end{rmk}
	
	\begin{rmk}\label{tight}
		It also applies to the case $\subman=\R^m$, with $\bsubman\subset\R^m$ a closed embedded submanifold:
		the lower bounds obtained in \cref{mono.sec} (see also the proof of \cref{quant.phi})
		show that the varifolds induced by $\Phi_k$ form a tight sequence, and the result then follows with the same proofs.
	\end{rmk}
	
	As for the regularity of the limit, we have the following.
	
	\begin{thm}\label{reg.thm.intro}
		For a parametrized free boundary stationary varifold $(\tilde\Sigma,\Phi,N)$, the map $\Phi$ is smooth up to the boundary $\de\Sigma$, where $\de_\nu\Phi\perp T\bsubman$. Also, on the components of $\tilde\Sigma$ where $\Phi$ is not (a.e.) constant, the multiplicity $N$ is constant and $\Phi$ is a branched minimal immersion outside $\de\Sigma$.
	\end{thm}
	
	\begin{rmk}
		We stress that the limit (branched) immersion $\Phi$ is free boundary minimal in the sense that it meets the constraint $\bsubman$ orthogonally along $\de\tilde\Sigma$.
		However, there could be points $x$ in the interior $\operatorname{int}(\Sigma)=\Sigma\setminus\de\Sigma$ with
		$\Phi(x)\in\bsubman$---a possibility which cannot happen, e.g., for $(\bar B^3,S^2)$ (on the components where $\Phi$ is not constant); unlike the main result of \cite{lz1}, at such points the orthogonality is not guaranteed.
	\end{rmk}
	
	A simple corollary is, for instance, the following. Note that other min-max situations can be dealt with in the same way.
	
	\begin{corollary}\label{minmax}
		Given any collection $\mathcal{F}$ of compact subsets of the space of smooth immersions $(\Sigma,\de\Sigma)\to(\subman,\bsubman)$, assuming $\mathcal{F}$ to be stable for isotopies of this space, the min-max value
		\begin{align*}
			&\beta:=\inf_{A\in\mathcal{F}}\max_{\Phi\in A}\operatorname{area}(\Phi)
		\end{align*}
		is the sum of the areas of finitely many free boundary minimal (branched) immersions $\Phi_{(i)}:\Sigma_{(i)}\to\subman$, whose domains are connected and have $\chi(\Sigma_{(i)})\ge \chi(\Sigma)$ and $g(\Sigma_{(i)})\le g(\Sigma)$.
	\end{corollary}
	
	\begin{rmk}
		While we deal only with surfaces $\Sigma$, the proofs in the next two sections generalize immediately to closed $k$-dimensional domains, with the energy
		\begin{align*}
			&\Phi\mapsto \text{$k$-area}(\Phi)+\int_\Sigma|\ff^\Phi|^p\,\operatorname{vol}_\Phi
		\end{align*}
		for $W^{2,p}$ immersions $\Phi:\Sigma\to\subman$, with $p>k$.
		We get a stationary varifold in the limit, which is rectifiable since its density is bounded below, by \cref{mono.log}, \cref{mono.log.nobdry} and the arguments from \cref{quant.phi} (which carry over with obvious changes). 
		
		However, what seems so far out of reach (when $k>2$) is how to retain a \emph{parametrized} structure for the limit varifold.
	\end{rmk}
	
	\subsection{Organization of the paper}
	We conclude the introduction with a very brief description of the structure of the paper.
	\begin{itemize}
		\item In \cref{crit.sec} we show how to deduce \cref{minmax} from \cref{main.thm}, by introducing a Finsler manifold of maps and checking that it satisfies the conditions guaranteeing that Struwe's monotonicity trick applies;
		\item in \cref{firstvar.sec} we compute the first variation of $E_\sigma$ for special variations $X(\Phi)$,
		and use the resulting formula to show that the varifolds induced by the maps $\Phi_k$ converge, up to subsequences, to a free boundary stationary varifold;
		\item \cref{mono.sec} is devoted to the proof of the lower bound for the area mentioned earlier, in various forms;
		\item in \cref{asympt.sec} we show several structure results for the (weak) limit of the area measure that $\Phi_k$ induces on $\Sigma$ and we obtain \cref{main.thm}, under the assumption that $\Phi_k$ induces a constant conformal structure on $\Sigma$
		and ignoring possible concentration points for the area;
		\item in \cref{deg.sec} we remove the above assumption, studying carefully how to deal with all possible situations of degeneration of the conformal structure and describing how to recover the energy arising from concentration points, thus proving \cref{main.thm} in general;
		\item finally, \cref{reg.sec} is devoted to the regularity part, namely the proof of \cref{reg.thm.intro}.
	\end{itemize}

	\section{Almost critical points for \texorpdfstring{$E_\sigma$}{the energy}}\label{crit.sec}
	Let $(\subman^m,g)$ be a closed Riemannian manifold and $\bsubman^n\subset\subman$ a closed embedded submanifold, with $1\le n<m$.
	For simplicity, we will assume without loss of generality that $\subman$ is isometrically embedded in some Euclidean space $\R^\envdim$,
	although the proofs could be easily modified so as to avoid the Nash embedding theorem.
	
	Also, let $\Sigma$ be a compact surface, possibly with boundary $\de\Sigma$.
	In this paper we will study the following relaxation of the area functional: given an immersion $\Phi:\Sigma\to\subman$, we let
	\begin{align}\label{e.sigma}\begin{aligned}
		E_\sigma(\Phi)
		&:=\operatorname{area}(\Phi)
		+\sigma\operatorname{length}(\Phi|_{\de\Sigma})
		+\sigma^4\int_\Sigma|\ff^\Phi|^4\,\operatorname{vol}_{\Phi} \\
		&\phantom{:}=\int_\Sigma \operatorname{vol}_{\Phi}
		+\sigma\int_{\de\Sigma}\operatorname{vol}_{\Phi|_{\de\Sigma}}
		+\sigma^4\int_\Sigma|\ff^\Phi|^4\,\operatorname{vol}_{\Phi}.
	\end{aligned}\end{align}
	Here $\operatorname{vol}_{\Phi}$ and $\operatorname{vol}_{\Phi|_{\de\Sigma}}$ are the (two- and one-dimensional) volume forms of the induced metric $\Phi^*g$ on $\Sigma$ and $\de\Sigma$, which we will often identify with the corresponding measures.
	In the last term, $\ff^\Phi$ denotes the second fundamental form of $\Phi$.
	
	In order to construct almost critical maps for $E_\sigma$, with the constraint $\Phi(\de\Sigma)\subseteq\bsubman$, we introduce the topological space
	\begin{align*}
		&\bman:=\{\Phi\in W^{2,4}(\Sigma,\subman):\Phi\text{ is an immersion and }\Phi(\de\Sigma)\subseteq\bsubman\},
	\end{align*}
	with the topology induced from $W^{2,4}(\Sigma,\subman)$, in turn induced from $W^{2,4}(\Sigma,\R^\envdim)$.
	Recall that $W^{2,4}(\Sigma,\R^\envdim)$ embeds into $C^1(\Sigma,\R^\envdim)$, so that the definition makes sense and $\bman$ is canonically a Banach manifold.
	
	For each $\Phi\in\bman$, the tangent space $T_\Phi\bman$ identifies with the Banach space of $W^{2,4}$ sections $s:\Sigma\to T\subman$ of the pullback bundle $\Phi^*T\subman$, with $s\in T\bsubman$ along $\de\Sigma$.
	
	Given $\Phi\in\bman$, we call $g_\Phi:=\Phi^*g$ the metric that $\Phi$ induces on $\Sigma$.
	We endow $T_\Phi\bman$ with the following norm: we let
	\begin{align*}
		&\|s\|_{\Phi}:=\|s\|_{L^\infty}+\|\nabla s\|_{L^\infty}+\|\nabla^2 s\|_{L^4},
	\end{align*}
	where $\nabla$ is the pullback connection on $\Phi^*T\subman$
	and the norms are with respect to the metrics $g$ on $T\subman$ and $g_\Phi$ on $T^*\Sigma$.
	It is straightforward to check that this choice satisfies the requirements to be a Finsler structure on $\bman$ (see \cite[p.~54]{ghou} for the definition).
	
	\begin{proposition}
		The Finsler manifold $\bman$ is complete.
	\end{proposition}
	
	Recall that the distance between two elements $\Phi_1,\Phi_2\in\bman$ (in the same connected component) is defined to be the infimum of
	$\int_0^1\|\dot\gamma(t)\|_{\gamma(t)}\,dt$, as $\gamma:[0,1]\to\bman$ ranges among all piecewise $C^1$ curves from $\Phi_1$ to $\Phi_2$. It is a consequence of the Finsler structure axioms that it induces the original topology on $\bman$.

	\begin{proof}
		Let $(\Phi_k)_{k\ge 0}$ be a Cauchy sequence. Up to subsequences, we can assume that $\sum_k \operatorname{dist}(\Phi_k,\Phi_{k+1})<\infty$.
		Hence, by definition we can find a piecewise $C^1$ curve $\Phi:[0,\infty)\to\bman$ of finite length, with $\Phi(k)=\Phi_k$ for every $k\in\N$. We will use the notation $\Phi_t$ in place of $\Phi(t)$. It suffices to show that $\Phi_t$ converges in $W^{2,4}$ as $t\to\infty$. With a perturbation argument, we can assume that $\Phi_t(x)$ is smooth in the couple $(x,t)$.
		
		Let $w_t:=\frac{d\Phi_t}{dt}$. Since $w_t$ is bounded pointwise by the summable (in $t$) quantity $\|w_t\|_{\Phi_t}$, we know that $\Phi_t$ converges in $C^0$ to a limit $\Phi_\infty$.
		
		Let $g_t:=g_{\Phi_t}$ be the metric induced by the immersion $\Phi_t$ on $\Sigma$. For a fixed $v\in T\Sigma$ we have
		\begin{align*}
			&\frac{d}{dt}g_t(v,v)
			=\frac{d}{dt}|d\Phi_t[v]|^2
			=2\ang{d\Phi_t[v],\nabla_v w_t}
		\end{align*}
		and, since $|\nabla_v w_t|
		\le\|w_t\|_{\Phi_t}|v|_{g_t}$, we deduce that
		\begin{align*}
			&\Big|\frac{d}{dt}g_t(v,v)\Big|\le 2g_t(v,v)\|w_t\|_{\Phi_t}.
		\end{align*}
		Hence, for $v\neq 0$, the time derivative of $\log g_t(v,v)$ is bounded in $L^1$ on $[0,\infty)$. Thus there exists a constant $C>0$ such that
		\begin{align*}
			&C^{-2}g_0(v,v)\le g_t(v,v)\le C^2g_0(v,v)
		\end{align*}
		for all $t\ge 0$ and all $v\in T\Sigma$. As a consequence, for any $x\in\Sigma$ and any $v\in T_x\Sigma$
		\begin{align*}
			&|\nabla_{\de_t}(d\Phi_t[v])|
			\le |\nabla w_t|_{g_t}|v|_{g_t}
			\le C\|w_t\|_{\Phi_t}|v|_{g_0},
		\end{align*}
		with $\nabla_{\de_t}$ being the covariant derivative along the curve $\Phi_t(x)$. Together with the $C^0$ convergence $\Phi_t\to\Phi_\infty$, this implies that actually $\Phi_t\to\Phi_\infty$ in $C^1$. Finally, given smooth vector fields $X,Y$ on $\Sigma$,
		\begin{align*}
			&\nabla_{\de_t}\nabla_{X}(d\Phi_t[Y])
			=\nabla_X\nabla_Y w_t+\operatorname{Rm}(d\Phi_t[X],w_t)(d\Phi_t[Y])
		\end{align*}
		where $\riem(V,W)Z=\nabla^2_{W,V}Z-\nabla^2_{V,W}Z$ is the Riemann tensor of $\subman$. Again, thanks to the comparability between $g_0$ and $g_t$, the right-hand side is bounded in $L^4$ by $\|w_t\|_{\Phi_t}$, up to a multiplicative constant depending only on $X,Y$. This implies the convergence $\Phi_t\to\Phi_\infty$ in $W^{2,4}$.
	\end{proof}
	
	The following variational result, essentially due to Struwe, is proved in \cite{riv.notes}. Before stating it, we give a notion of \emph{admissible family}.
	
	\begin{definition}
		Given a Banach manifold $\bman$, a nonempty family $\mathcal{F}$ of subsets of $\bman$ is said to be \emph{admissible}
		if, for any continuous deformation $F:[0,1]\times\bman\to\bman$ with $F_0=\operatorname{id}_{\bman}$ and $F_t$ a homeomorphism for all $0\le t\le 1$, we have $F_1(A)\in\mathcal{F}$ for all $A\in\mathcal{F}$ (where $F_t:=F(t,\cdot)$).
	\end{definition}
	
	\begin{proposition}\label{struwe}
		Assume $(E_\sigma)_{\sigma\ge 0}$ is a family of $C^1$ functionals on a complete Finsler manifold $\bman$, with $E_\sigma(x)$ differentiable in $\sigma$ and $\sigma\mapsto E_\sigma(x)$, $\sigma\mapsto\frac{d}{d\sigma}E_\sigma(x)$ both increasing in $\sigma$, for every $x\in\bman$. Assume also that
		\begin{align}\label{struwe.tech}
			&\|dE_{\sigma_j}(x_j)-dE_\sigma(x_j)\|\to 0
		\end{align}
		whenever
		$1\ge\sigma_j\ge\sigma>0$, $\sigma_j\to\sigma$ and $\limsup_{j\to\infty}E_\sigma(x_j)<\infty$.
		
		Then, for any admissible family $\mathcal{F}$, defining the min-max values
		\begin{align*}
			&\beta(\sigma):=\inf_{A\in\mathcal{F}}\sup_{x\in A}E_\sigma(x),
		\end{align*}
		there exist sequences $(\sigma_k)\subseteq (0,1)$ and $(x_k)\subseteq\bman$, with $\sigma_k\to 0$, such that
		\begin{align*}
			&E_{\sigma_k}(x_k)-\beta(\sigma_k)\to 0,
			\quad\|dE_{\sigma_k}(x_k)\|<f(\sigma_k),
			\quad\sigma_k\log(1/\sigma_k) \frac{d}{d\sigma}E_{\sigma}(x_k)\Big|_{\sigma_k}\to 0,
		\end{align*}
		where $f:(0,\infty)\to(0,\infty)$ is any function fixed in advance.
	\end{proposition}

	This statement is quite robust and can be adapted to other kinds of min-max problems, where one replaces admissible families with other notions.

	\begin{rmk}
		Actually, in \cite{riv.notes} the functional $E_\sigma$ is assumed to be Palais--Smale, and the second conclusion becomes $dE_{\sigma_k}(x_k)=0$. Without this hypothesis, we can still find almost critical points $x_k$ for $E_{\sigma_k}$,
		in the sense that we can require $\|dE_{\sigma_k}(x_k)\|$ to be as small as we want, with the same proof.
	\end{rmk}
	
	\begin{proposition}\label{struwe.sat}
		The functionals $(E_\sigma)_{\sigma\ge 0}$ previously defined satisfy the assumptions of \cref{struwe}.
	\end{proposition}
	
	Before proving this fact, we make an important observation.
	
	\begin{proposition}\label{intrins}
		For $X,Y$ vector fields on $\Sigma$ we have $(\nabla d\Phi)(X,Y)=\ff^\Phi(\Phi_*X,\Phi_*Y)$.
	\end{proposition}
	
	\begin{proof}
		The left-hand side equals $\nabla_X(\Phi_*Y)-\Phi_*\nabla_X Y$;
		since $\Phi$ is an isometry from $(\Sigma,g_\Phi)$ to the immersed surface $\Phi$,
		the term $\Phi_*\nabla_X Y$ equals (locally) the Levi-Civita connection $\nabla_{\Phi_*X}\Phi_*Y$
		on this surface; the latter equals the orthogonal projection of $\nabla_X(\Phi_*Y)$ onto the tangent plane, since $\nabla$ is the pullback of the Levi-Civita connection from $\subman$.
	\end{proof}
	
	\begin{proof}[Proof of \cref{struwe.sat}]
		We only need to check that \cref{struwe.tech} holds. We first show how to obtain an upper bound for $|dE_{\sigma'}(\Phi)[w]-dE_\sigma(\Phi)[w]|$, when $1\ge{\sigma'}\ge\sigma>0$.
		
		If $\Phi\in\bman$ is a smooth map and $(\Phi_t)$ is a smooth variation (with $\Phi_0=\Phi$), we compute
		\begin{align*}
			&\frac{d}{dt}\int_\Sigma\operatorname{vol}_{\Phi_t}\Big|_{t=0}
			=\int_\Sigma\ang{d\Phi,\nabla w}\,\operatorname{vol}_\Phi,
		\end{align*}
		where $w:=\frac{d}{dt}\Phi_t\Big|_{t=0}$ belongs to $T_\Phi\bman$,
		and the scalar product (with respect to $g_\Phi$) in the integral is bounded by $2\|\nabla w\|_{L^\infty}\le 2\|w\|_\Phi$.
		
		With a similar computation for the length of $\Phi|_{\de\Sigma}$, we get
		\begin{align}\label{first.var.halfway}\begin{aligned}
			\frac{d}{dt}E_\sigma(\Phi)\Big|_{t=0}
			&=\int_\Sigma(1+\sigma^4|\ff^\Phi|^4)\ang{d\Phi,\nabla w}\,\operatorname{vol}_\Phi
			+\sigma\int_{\de\Sigma}\ang{d\Phi[\tau],\nabla_\tau w}\,\operatorname{vol}_{\Phi|_{\de\Sigma}} \\
			&\quad +\sigma^4\int_\Sigma \frac{d}{dt}|\ff^{\Phi_t}|^4\Big|_{t=0}\,\operatorname{vol}_\Phi,
		\end{aligned}\end{align}
		where $\tau$ is the unit vector (with respect to $g_\Phi$) orienting $\de\Sigma$.
		
		Given a local orthonormal frame $\{e_1,e_2\}$, oriented as $\Sigma$, define $n_t:=d\Phi_t[e_1]\wedge d\Phi_t[e_2]$.
		We have $|\ff^{\Phi_t}|=|\nabla n_t|$ and
		\begin{align*}
			&\nabla_{\de_t}\nabla_X n_t
			=\nabla_X\nabla_{\de_t}n_t+\riem\Big(d\Phi[X],\frac{d\Phi}{dt}\Big)n_t,
		\end{align*}
		where $\riem(a,b)(c\wedge d):=(\riem(a,b)c)\wedge d+c\wedge(\riem(a,b)d)$ for vectors in $T\subman$. At $t=0$ the above equals $\nabla_X\omega+R(d\Phi[X],w)n$, where $n:=n_0$ and
		\begin{align*}
			&\omega:=\nabla_{e_1}w\wedge d\Phi[e_2]+d\Phi[e_1]\wedge\nabla_{e_2}w-\ang{d\Phi[e_i],\nabla_{e_i} w}n.
		\end{align*}
		Using \cref{intrins} we see that $|\nabla_X\omega|\le C|X|(|\nabla^2 w|+|\nabla w||\ff^{\Phi}|)$.
		
		Finally, the contribution of the metric $g_{\Phi_t}$ for the time derivative of $|\nabla n_t|^4$ is just $-4|\nabla n|^2\ang{d\Phi\otimes\nabla w,\nabla n\otimes\nabla n}$. Combining this fact with the preceding computations, we deduce that the time derivative of $|\ff^{\Phi_t}|^4$ at $t=0$ is bounded by
		\begin{align}\label{ff.contrib}
			&|\ff^\Phi|^3|\nabla^2 w|+|\ff^\Phi|^4|\nabla w|
			+|\ff^\Phi|^3|w|
		\end{align}
		up to a multiplicative constant depending on $\subman$.
		
		Thus, using \cref{first.var.halfway}, \cref{ff.contrib}, H\"older's inequality and Young's inequality, we see that
		\begin{align*}
			&|dE_{\sigma'}(\Phi)[w]-dE_\sigma(\Phi)[w]|
			\le C\frac{{\sigma'}-\sigma}{\sigma}E_\sigma(\Phi)\|w\|_\Phi
			+C({\sigma'}-\sigma)E_\sigma(\Phi)^{3/4}\|w\|_\Phi
		\end{align*}
		for $0<\sigma\le{\sigma'}\le 2\sigma$.
		Since $E_\sigma$ and $E_{\sigma'}$ are $C^1$ functionals, this bound holds for general $\Phi\in\bman$ and $w\in T_\Phi\bman$.
		Starting from this estimate, it is immediate to check that \cref{struwe.tech} is satisfied.
	\end{proof}
	
	Thanks to \cref{struwe.sat}, letting $f(\sigma):=\sigma^5$ we can then find sequences of numbers $\sigma_k\to 0$ and maps $\Phi_k\in\bman$ satisfying the conclusions of \cref{struwe}.
	In particular,
	\begin{align}\label{almost.crit}
		&\|dE_{\sigma_k}(\Phi_k)\|_{\Phi_k}<\sigma_k^5
	\end{align}
	and
	\begin{align}\label{entropy}
		&\sigma_k\log(1/\sigma_k)\operatorname{length}(\Phi_k|_{\de\Sigma})
		+\sigma_k^4\log(1/\sigma_k)\int_\Sigma|\ff^{\Phi_k}|^4\,\operatorname{vol}_{\Phi_k}
		\to 0.
	\end{align}
	Since smooth functions are dense in $\bman$, we can assume that the maps $\Phi_k$ are smooth.
	
	In the following sections we will study the limit behavior of the measures $\nu_k:=\operatorname{vol}_{\Phi_k}$
	and the varifolds $\vfd_k$ induced by $\Phi_k$. Note that the weight measure $|\vfd_k|$ equals $(\Phi_k)_*\nu_k$.
	
	We conclude this section by discussing how \cref{minmax} follows from \cref{main.thm}.
	
	\begin{proof}[Proof of \cref{minmax}]
		For any $A\in\mathcal F$, by compactness of $A$ we have
		\begin{align*}
			&\max_{\Phi\in A}E_\sigma(\Phi)\to\max_{\Phi\in A}\operatorname{area}(\Phi)\quad\text{as }\sigma\to 0.
		\end{align*}
		Hence, the min-max value $\beta(\sigma)$ for $E_\sigma$ converges to $\beta$.
		Although $\mathcal{F}$ is not stable under isotopies of $\bman$, \cref{struwe} still applies since in its proof we can use a pseudo-gradient flow preserving the subset of smooth immersions. Taking then smooth maps $\Phi_k$ as above,
		the statement follows from \cref{main.thm}, \cref{reg.thm.intro} and the fact that
		\begin{align*}
			&\lim_{k\to\infty}\operatorname{area}(\Phi_k)=\lim_{k\to\infty}\beta(\sigma_k)=\beta. \qedhere
		\end{align*}
	\end{proof}

	\section{First variation}\label{firstvar.sec}
	In this section we will derive a particularly useful formula for the first variation of $E_\sigma$ at $\Phi\in\bman$, for infinitesimal variations $w\in T_\Phi\bman$ of the form $X(\Phi)$, with $X$ a smooth vector field on $\subman$.
	
	Let $\Phi\in\bman$ be a smooth map and $w\in T_\Phi\bman$ a smooth section of $\Phi^*T\subman$, with $w\in T\bsubman$ on $\de\Sigma$.
	In the sequel, $\{e_1,e_2\}$ will be an oriented orthonormal basis at an arbitrary point of $\Sigma$, with respect to the induced metric $g_\Phi$.
	The $(1,1)$-tensor $J:T\Sigma\to T\Sigma$, given by $Je_1:=e_2$ and $Je_2:=-e_1$, is parallel for this metric.
	
	As in the proof of \cref{struwe.sat}, we use the notation $n:=\Phi_*e_1\wedge\Phi_*e_2$ and we set $f:=|\ff^\Phi|^2=|\nabla n|^2$.
	We also define the sections $\hat I$ and $\hat J$ of $\Phi^*T\subman\otimes T^*\Sigma$, as well as the section
	$\hat\ff$ of $\Phi^*T\subman\otimes T^*\Sigma\otimes T^*\Sigma$, by
	\begin{align*}
		&\hat I(v):=\Phi_*v,\ \hat J(v):=\Phi_*(Jv),\ \hat\ff(v,v'):=\ff^{\Phi}(\Phi_*v,\Phi_*v'),\quad\text{for }v,v'\in T\Sigma.
	\end{align*}
	
	Recall the following formula, which was computed in that proof:
	\begin{align}\label{first.var}\begin{aligned}
		dE_\sigma(\Phi)[w]
		&=\int_\Sigma (1+\sigma^4 f^2)\Bang{\hat I,\nabla w}
		+\sigma\int_{\de\Sigma}\ang{\Phi_*\tau,\nabla_\tau w} \\
		&\quad +4\sigma^4 \int_\Sigma f\ang{\nabla n,\nabla\omega+\riem(d\Phi,w)n}
		-4\sigma^4\int_\Sigma f\Bang{\hat I\otimes\nabla w,\nabla n\otimes\nabla n},
	\end{aligned}\end{align}
	where we omit the volume forms and $\omega$ denotes the infinitesimal variation of $n$, namely
	\begin{align*}
		&\omega=\nabla_{e_i}w\wedge\hat J(e_i)-\Bang{\nabla w,\hat I}n.
	\end{align*}
	When the variation $w$ has the form $w=X(\Phi)$, using \cref{intrins} we get
	\begin{align}\label{nabla.sq}
	\begin{aligned}
		\nabla w&=\nabla X(\Phi)[\Phi_*\cdot]=\nabla X\circ\hat I, \\
		\nabla_{e_i,e_j}^2 w&=\nabla^2 X(\Phi)[\Phi_*e_i,\Phi_*e_j]+\nabla X(\Phi)[\hat\ff(e_i,e_j)].
	\end{aligned}
	\end{align}

	For such special variations, \cref{first.var} becomes
	\begin{align}\label{first.var.X}\begin{aligned}
		dE_\sigma(\Phi)[w]&=\int_\Sigma (1+\sigma^4 f^2)\ang{\hat I,\nabla X\circ\hat I}
			+\sigma\int_{\de\Sigma}\ang{\Phi_*\tau,\nabla X[\Phi_*\tau]} \\
		&\quad+4\sigma^4\int_\Sigma f(\ang{\nabla n,\nabla\omega}+\ang{\nabla n,\riem(d\Phi,X(\Phi))n}) \\		
		&\quad-4\sigma^4\int_\Sigma f\Bang{\hat I\otimes(\nabla X\circ\hat I),\nabla n\otimes\nabla n}.
	\end{aligned}\end{align}	
	
	We now write the term $\ang{\nabla n,\nabla\omega}$ in a way which will prove useful for our later work. Since $\ang{\nabla_{e_i}n,n}=0$, we compute
	\begin{align*}
		\ang{\nabla n,\nabla\omega}&=\Bang{\nabla n,\nabla(\nabla_{e_i}w\wedge\hat J(e_i))}-\Bang{\nabla w,\hat I}|\nabla n|^2
	\end{align*}
	and the first term equals $\Bang{\nabla_{e_j}n,\nabla^2_{e_j,e_k}w\wedge\hat J(e_k)+\nabla_{e_k}w\wedge\hat\ff(e_j,Je_k)}$.
	Substituting the above formulas for $\nabla w$ and $\nabla^2 w$, we get
	\begin{align*}
		\ang{\nabla n,\nabla\omega}&=\langle\nabla_{e_j}n,\nabla^2 X[\Phi_*e_j,\Phi_*e_k]\wedge\hat J(e_k)+\nabla X[\hat\ff(e_j,e_k)]\wedge\hat J(e_k) \\
		&\quad +\nabla X[\Phi_*e_k]\wedge\hat\ff(e_j,Je_k)\rangle-\Bang{\nabla X,\hat I}|\nabla n|^2.
	\end{align*}
	Thus,
	\begin{align}\label{R.bound}
		&f|\ang{\nabla n,\nabla\omega}|\le C(\subman)(\|\nabla X\|_{L^\infty}f^2+\|\nabla^2 X\|_{L^\infty}f^{3/2}).
	\end{align}
	
	We are now ready to state an initial consequence of this bound.
	
	\begin{definition}\label{fb.def}
		A $k$-varifold $\vfd$ on $\subman$ is a \emph{free boundary stationary varifold} for the couple $(\subman,\bsubman)$ if it holds that
		\begin{align*}
			&\frac{d}{dt}\|(F_t)_*\vfd\|(\subman)\Big|_{t=0}=0
		\end{align*}
		whenever $(F_t)_{-\epsilon<t<\epsilon}$ is a family of diffeomorphisms of $\subman$ with $F_t(\bsubman)=\bsubman$, $F_0=\operatorname{id}$
		and $F_t(x)$ smooth in the couple $(t,x)$. We say that $\vfd$ is free boundary stationary \emph{outside} a closed set $K\subseteq\subman$ if the same holds for isotopies $(F_t)$ such that $F_t|_U=\operatorname{id}$ for some neighborhood $U\supseteq K$.
	\end{definition}
	
	\begin{definition}
		We denote $\fbvf$ the linear space of smooth vector fields $X$ on $\subman$ which are tangent to $\bsubman$,
		namely such that $X(p)\in T_p\bsubman$ for all $p\in\bsubman$.
	\end{definition}
	
	\begin{rmk}
		With $X:=\frac{d}{dt}F_t\Big|_{t=0}$, we have $X\in\fbvf$ and
		\begin{align*}
			&\frac{d}{dt}\|(F_t)_*\vfd\|(\subman)\Big|_{t=0}=\int_{(p,\Pi)\in\operatorname{Gr}_k(\subman)}\operatorname{div}_{\Pi}X\,d\vfd(p,\Pi),
		\end{align*}
		where $\operatorname{Gr}_k(\subman)$ is the Grassmannian bundle made of couples $(p,\Pi)$ with $p\in\subman$ and $\Pi\subseteq T_p\subman$ a $k$-plane. Conversely, given $X$ tangent to $\bsubman$, we can take $F_t$ to be its flow. Hence,
		$\vfd$ is a free boundary stationary varifold if and only if
		\begin{align*}
			&\int_{(p,\Pi)\in\operatorname{Gr}_k(\subman)}\operatorname{div}_{\Pi}X\,d\vfd(p,\Pi)=0
			\quad\text{for all }X\in\fbvf.
		\end{align*}
		Similarly, $\vfd$ is free boundary stationary outside $K$ if and only if the same holds for all $X\in\fbvf\cap C^\infty_c(\subman\setminus K)$.
	\end{rmk}
	
	Given a sequence $(\Phi_k)$ as in \cref{crit.sec}, the following holds.
	
	\begin{thm}\label{limit.is.stat}
		The varifolds $\vfd_k$ induced by $\Phi_k$ converge, up to subsequences, to a free boundary stationary varifold $\vfd_\infty$.
	\end{thm}
				
	A priori it is not clear whether $\vfd_\infty$ is integer rectifiable. This, together with a structure theorem for $\vfd_\infty$, will be proved later on.
	
	\begin{proof}
		Fix any $(F_t)_{-\epsilon<t<\epsilon}$ as above and consider the variation $(F_t\circ\Phi_k)\subseteq\bman$.
		The corresponding infinitesimal variation $w_k\in T_{\Phi_k}\bman$ is just $w_k=X\circ\Phi_k$. Hence, \cref{first.var.X} and \cref{almost.crit} give
		\begin{align}\label{first.var.X.k}\begin{aligned}
			&\int_\Sigma (1+\sigma_k^4 f_k^2)\ang{\hat I_k,\nabla X\circ\hat I_k}
			+\sigma_k\int_{\de\Sigma}\ang{(\Phi_k)_*\tau,\nabla X[(\Phi_k)_*\tau]} \\
			&+4\sigma_k^4\int_\Sigma f_k(\ang{\nabla n_k,\nabla\omega_k}+\ang{\nabla n_k,\riem(d\Phi_k,X(\Phi_k))n_k}) \\		
			&-4\sigma_k^4\int_\Sigma f_k\Bang{\hat I_k\otimes(\nabla X\circ\hat I_k),\nabla n_k\otimes\nabla n_k} \\
			&=o(\sigma_k^4\|w_k\|_{\Phi_k}).
		\end{aligned}\end{align}
		We now show that all terms where $\sigma_k$ appears are infinitesimal as $k\to\infty$.
		Note that $\big|\Bang{\hat I_k,\nabla X\circ\hat I_k}\big|\le 2\|\nabla X\|_{L^\infty}$, since the scalar product is with respect to the induced metric $g_{\Phi_k}$.
		Hence, by \cref{entropy},
		\begin{align*}
			&\sigma_k^4\int_\Sigma f_k^2\Bang{\hat I_k,\nabla X\circ\hat I_k}\to 0
		\end{align*}
		and similarly the boundary term is also infinitesimal.
		Thanks to the boundedness of the area of $\Phi_k$, the pointwise bound \cref{R.bound} and H\"older's inequality,
		we deduce that also the remaining terms in the left-hand side of \cref{first.var.X.k} are infinitesimal,
		except for the first one.
		
		We now estimate $\|w_k\|_{\Phi_k}$. Note first that $|w_k|\le\|X\|_{L^\infty}$ and $|\nabla w_k|\le\|\nabla X\|_{L^\infty}$. Also, from \cref{nabla.sq} we get
		\begin{align*}
			&|\nabla^2 w_k|\le\|\nabla^2 X\|_{L^\infty}+\|\nabla X\|_{L^\infty}|\ff^{\Phi_k}|.
		\end{align*}
		We deduce that $\sigma_k\|w_k\|_{\Phi_k}\to 0$.
		
		Finally, $\Bang{\hat I_k,\nabla X\circ\hat I_k}(x)=\operatorname{div}_{(\Phi_k)_*[T_x\Sigma]}X$, so that
		\begin{align*}
			&\int_\Sigma\Bang{\hat I_k,\nabla X\circ\hat I_k}
			=\int_{(p,\Pi)\in\operatorname{Gr}_2(\subman)}\operatorname{div}_{\Pi}X\,d\vfd_k(p,\Pi)
		\end{align*}
		and, taking any subsequential limit $\vfd_\infty$, the claim follows.
	\end{proof}

	\section{A lower bound for the area}\label{mono.sec}
	In order to obtain more information for the asymptotic behavior of the measures $\nu_k$ and the varifolds $\vfd_k$ introduced at the end of \cref{crit.sec}, we first obtain (various versions of) a lower bound on the mass $\frac{|\vfd_k|(B_r(p))}{r^2}$. The main idea will be to mimick the proof of the monotonicity formula for stationary varifolds; since that proof uses vector fields in the ambient $\subman$, we will be able to use formula \cref{first.var.X}, involving variations of the form $X(\Phi)$.
	
	The statements contained in this section make it essential to require the decays $\sigma_k^4\log\sigma_k^{-1}\int_\Sigma f_k^2\,\operatorname{vol}_{\Phi_k}\to 0$,
	as well as $\sigma_k\log\sigma_k^{-1}\operatorname{length}(\Phi_k|_{\de\Sigma})\to 0$,
	guaranteed by \cref{struwe}.
	
	Rather than dealing with the sequence $(\Phi_k)$, in this section all the statements concern a general smooth map $\Phi\in\bman$, with a fixed value of $\sigma$. Of course, in order for the results to be useful in the asymptotic analysis, the constants appearing in their statements will depend neither on $\Phi$ nor on $\sigma$.
	
	
	\begin{definition}
		In the following statements, we say that a smooth map $\Phi\in\bman$ is \emph{$\epsilon$-critical} for $E_\sigma$ if $\|dE_\sigma(\Phi)\|_\Phi\le\epsilon$, meaning that $|dE_\sigma(\Phi)[w]|\le\epsilon\|w\|_\Phi$ for all $w\in T_\Phi\bman$.
	\end{definition}
	
	\begin{proposition}\label{mono.log}
		Let $\Phi$ be $\sigma^5$-critical for $E_\sigma$, $x\in\Sigma$, and denote $p:=\Phi(x)$.
		Assume $U\subseteq\Sigma$ is an open neighborhood of $x$.
		Defining the measures $\mu:=(\Phi|_U)_*(\operatorname{vol}_\Phi)$ and $\lambda:=(\Phi|_{\de\Sigma\cap U})_*(\sigma\operatorname{vol}_{\Phi|_{\de\Sigma}})+(\Phi|_U)_*(\sigma^4 f^2\,\operatorname{vol}_\Phi)$ on $\subman$, assume also that
		\begin{align*}
			&\lambda(B_s(p))\le\delta\mu(B_{5s}(p))\quad\text{for all radii }s>0,
		\end{align*}		
		for some $0<\delta<1$.
		Given $r>s\ge\sigma$, if $B_r(p)\cap\Phi(\de U)=\emptyset$ then we have
		\begin{align*}
			&\frac{\mu(B_r(p))}{r^2}
			\ge (c-C\delta\log(r/s))\frac{\mu(B_{s}(p))}{s^2}-C\sigma^2,
		\end{align*}
		for some constants $c,C>0$ depending on $\subman$ and $\bsubman$.
	\end{proposition}

	Note that $\de U$ is the topological boundary of $U$ in $\Sigma$ and therefore does not include $\de\Sigma\cap U$.
	Recall that $f=|\ff^\Phi|^2$.

	Before delving into the proof, we state without proof an immediate but useful fact.
%
%

	\begin{proposition}\label{coordinates}
		There exists a constant $c_F(\subman,\bsubman)$ such that, for every $p\in\subman$, there are coordinates
		\begin{align*}
			&\xi=(\xi_1,\dots,\xi_m):B_{c_F}(p)\to\R^m
		\end{align*}
		depending on the center $p$, satisfying
		\begin{align}\label{almost.flat}
			&g_{ij}(0)=\delta_{ij},
			\quad\|g_{ij}\|_{C^2}\le C(\subman,\bsubman),
			\quad \mz\operatorname{dist}(\cdot,p)\le|\xi|\le 2\operatorname{dist}(\cdot,p)
		\end{align}
		for the Euclidean metric $|\cdot|$.
		When $p\in\bsubman$ we also ask that the coordinates are adapted to $\bsubman$, in the sense that $B_{c_F}(p)\cap\bsubman$ corresponds to $\{\xi_{n+1}=\cdots=\xi_m=0\}$.
	\end{proposition}
	
	\begin{proof}[Proof of \cref{mono.log}]
		Without loss of generality, we can assume $r,\frac{\sigma}{s}\le c'$ for a constant $c'<c_F$ to be chosen later. Once we get the desired estimate with these constraints, the statement follows in general with possibly different values of $c$ and $C$.
		
		We will imitate the proof of the monotonicity formula, using now our equation \cref{first.var.X}. Assume first $B_{r}(p)\cap\bsubman=\emptyset$.
		In this case we can find coordinates $\xi:B_{r}(p)\to\R^m$ as in \cref{coordinates}.
		Given a decreasing cut-off function $\chi\in C^\infty_c([0,\infty))$, with $\chi=1$ on $[0,1/4]$ and $\chi=0$ on $[1/2,\infty)$, for $0<\tau<r$ we set
		$\chi_\tau:=\chi(|\xi|/\tau)$ and $X_\tau:=\chi_\tau \xi_i\frac{\de}{\de\xi_i}$.
		
		Note that, by \cref{almost.flat}, we have $|\nabla X_\tau|\le C$, $|\nabla^2 X_\tau|\le C\tau^{-1}$ and
		\begin{align}\label{diverg}
		&\operatorname{div}_\Pi(X_\tau)\ge (2-C\tau)\chi_\tau+(1+C\tau)\chi'(|\xi|/\tau)\frac{|\xi|}{\tau}
		\end{align}
		for any $p\in B_\tau(q)$ and any 2-plane $\Pi\subseteq T_p\subman$ (recall that $\chi'\le 0$).
		
		We now want to apply \cref{first.var.X} with the infinitesimal variation $w:=X_\tau(\Phi)\uno_U$, which is admissible since
		$X_\tau$ vanishes near $\Phi(\de U)$.
		By \cref{R.bound} we have
		\begin{align*}
			&\sigma^4 f\ang{\nabla n,\nabla\omega}\le C\sigma^4(f^2+\tau^{-1}f^{3/2})\uno_{\spt(w)};
		\end{align*}
		hence, the corresponding term in the first variation is bounded by
		\begin{align*}
			&C\lambda(B_\tau(p))+C\sigma\tau^{-1}\lambda(B_\tau(p))^{3/4}\mu(B_\tau(p))^{1/4}
			\le C(\delta+\sigma\tau^{-1})\mu(B_{5\tau}(p)).
		\end{align*}
		Similarly, the curvature term in \cref{first.var.X} is bounded by $C\sigma\mu(B_{5\tau}(p))$,
		while the last term is again bounded by $C\delta\mu(B_{5\tau}(p))$.
		Also, the boundary term vanishes since the support of $w$ does not intersect $\de\Sigma$.
		
		Finally, as in the proof of \cref{limit.is.stat}, we have
		\begin{align*}
			\|w\|_\Phi &\le \|X_\tau\|_{L^\infty}+\|\nabla X_\tau\|_{L^\infty}
			+\Big(\int_{\operatorname{spt}(w)}(\|\nabla^2 X_\tau\|_{L^\infty}+\|\nabla X_\tau\|_{L^\infty}f^{1/2})^4\,\operatorname{vol}_\Phi\Big)^{1/4} \\
			&\le C+C\tau^{-1}\mu(B_\tau(p))^{1/4}+\sigma^{-1}\lambda(B_\tau(p))^{1/4} \\
			&\le C+C(\tau^{-4}+\sigma^{-4})\mu(B_{5\tau}(p)),
		\end{align*}
		so that from \cref{almost.crit} we get $|dE_\sigma(\Phi)[w]|\le C\sigma^5+C\sigma\mu(B_{5\tau}(p))$ for $\tau\ge\sigma$.
		
		Hence, defining $h(\tau):=\tau^{-2}\int_U\chi_\tau(\Phi)\,\operatorname{vol}_\Phi$, \cref{first.var.X} and a straightforward computation give
		\begin{align}\label{final.diffeq}
			&h'(\tau)\ge -C(\delta+\sigma\tau^{-1}) \tau^{-3}\mu(B_{5\tau}(p))-C\tau^{-2}\mu(B_{5\tau}(p))-C\sigma^2
		\end{align}
		for $\tau\ge\sigma$. Call $\bar r$ the biggest radius in $[s,r]$ such that
		\begin{align*}
			&\frac{\mu(B_{\bar r}(p))}{\bar r^2}\ge \frac{\mu(B_{s}(p))}{s^2}.
		\end{align*}
		For $\frac{r}{5}\ge\tau\ge\bar r\ge\sigma$, \cref{final.diffeq} becomes
		\begin{align*}
			&h'(\tau)\ge -C(\delta \tau^{-1}+\sigma\tau^{-2}+1)\frac{\mu(B_{s}(p))}{s^2}-C\sigma^2.
		\end{align*}
		
		Integrating this inequality between $8\bar r$ and $\frac{r}{5}$ we get
		\begin{align*}
			&\frac{\mu(B_{r/5}(p))}{(r/5)^2}
			\ge h(r/5)
			\ge h(8\bar r)-C(\delta\log(r/s)+\sigma s^{-1}+r)\frac{\mu(B_{s}(p))}{s^2}-C\sigma^2 r,
		\end{align*}
		unless $r<40\bar r$, in which case the statement follows trivially. Since $r$ and $\frac{\sigma}{s}$ are both bounded by $c'$,
		observing that $h(8\bar r)\ge\frac{\mu(B_{\bar r}(p))}{64\bar r^2}\ge\frac{\mu(B_{s}(p))}{64s^2}$ we arrive at
		\begin{align*}
			&\frac{\mu(B_{r/5}(p))}{(r/5)^2}
			\ge\Big(\frac{1}{64}-C\delta\log(r/s)-2Cc'\Big)\frac{\mu(B_{s}(p))}{s^2}-C\sigma^2 r,
		\end{align*}
		and the statement follows in this case, once we impose $2Cc'<\frac{1}{64}$.
		
		If $r':=\operatorname{dist}(p,\bsubman)<r$, we let $q$ be a nearest point to $p$ in $\bsubman$ (hence, $q=p$ when $r'=0$).
		If $r'\ge s$, we know that the claim holds with $r'$ replacing $r$;
		so it follows also for $r$ if either $s\ge r/8$ or $r'\ge r/8$, with possibly different constants. Assume in the sequel that $r',s<r/8$.
		For $\tau>2r'+s$ we have
		\begin{align*}
			&\lambda(B_\tau(q))
			\le\lambda(B_{2\tau}(p))
			\le\delta\mu(B_{10\tau}(p))
			\le\delta\mu(B_{20\tau}(q)).
		\end{align*}
		So, using now coordinates centered at $q$ and adapted to $\bsubman$ and defining $h$ as before, we get
		\begin{align*}
			&h'(\tau)\ge -C(\delta \tau^{-1}+\sigma\tau^{-2}+1)\frac{\mu(B_{20\tau}(p))}{\tau^2}-C\sigma^2
		\end{align*}
		for $\tau\ge 2r'+s\ge\sigma$;
		note that now also the boundary term in \cref{first.var.X} is taken into account, giving again a contribution
		bounded by $C\tau^{-3}\lambda(B_\tau(q))\le C\delta\tau^{-3}\mu(B_{20\tau}(q))$ in the previous right-hand side.
		Similarly to the above,
		assume $2r'+s\le\bar r'\le r/2$ to be the smallest radius in this interval such that
		$\frac{\mu(B_\tau(q))}{\tau^2}\le\frac{\mu(B_s(p))}{s^2}$ for $\tau\in[\bar r',r/2]$;
		if such radius does not exist, then we have $\frac{\mu(B_{r/2}(q))}{(r/2)^2}\ge\frac{\mu(B_s(p))}{s^2}$
		and we are done thanks to the inclusion $B_r(p)\supseteq B_{r/2}(q)$.
		Integrating from $8\bar r'$ to $r/40$ (again, we can assume $\bar r'\le\frac{r}{320}$), we conclude that either
		\begin{align*}
			&\frac{\mu(B_{r/40}(q))}{(r/40)^2}\ge\Big(\frac{1}{64}-C\delta\log(r/s)-2Cc'\Big)\frac{\mu(B_{s}(p))}{s^2}-C\sigma^2,
		\end{align*}
		in which case we are done since $\mu(B_r(p))\ge\mu(B_{r/40}(q))$, or
		\begin{align*}
			&\frac{\mu(B_{r/40}(q))}{(r/40)^2}
			\ge\frac{\mu(B_{2r'+s}(q))}{64(2r'+s)^2}-(C\delta\log(r/s)+2Cc')\frac{\mu(B_s(p))}{s^2}-C\sigma^2.
		\end{align*}
		In this second case, if $r'<s$ then we use the inequality $\frac{\mu(B_{2r'+s}(q))}{(2r'+s)^2}\ge\frac{\mu(B_s(p))}{(3s)^2}$ and we are done.
		Otherwise, if $r'\ge s$ we use the inequality $\frac{\mu(B_{2r'+s}(q))}{(2r'+s)^2}\ge\frac{\mu(B_{r'}(p))}{(3r')^2}$
		and we conclude using the already obtained lower bound for this last ratio.
%
	\end{proof}
	
	\begin{rmk}\label{mono.easy}
		A similar choice of test vector fields gives the following monotonicity for general free boundary stationary varifolds $\vfd$: given $p$ in $\subman$, one has
		\begin{align}\label{mono.easy.zero}
			&\frac{|\vfd|(B_r(p))}{r^2}\ge(1+C(\subman,\bsubman)\sqrt r)^{-1}\frac{|\vfd|(B_s(p))}{s^2}
		\end{align}
		for $0<s<r<\operatorname{diam}(\subman)$ if $p\in\bsubman$, and for $0<s<r<\operatorname{dist}(p,\bsubman)$ otherwise.
		Indeed, it suffices to establish \cref{mono.easy.zero} assuming $r$ small, and also $s\ge\frac{r}{2}$, since for $s<\frac{r}{2}$ we can then compare dyadic radii $r,\frac{r}{2},\dots,2^{-k}r$ until $2^{-k-1}r\le s$.
		Pick coordinates as in \cref{coordinates}, with $|\operatorname{dist}(\cdot,p)-|\xi||\le C\operatorname{dist}(\cdot,p)^2$, and take now $\chi$ such that $\chi=1$ on $[0,1-2\sqrt r]$, $\chi=0$ on $[1-\sqrt r,\infty)$ and $|\chi'|\le Cr^{-1/2}$, so that $\chi_\tau$ is supported in $B_r(p)$ for $\tau\le r$. Setting $h(\tau):=\tau^{-2}\int_\subman\chi_\tau\,d|\vfd|$, the stationarity of $\vfd$ and \cref{diverg} then give
		\begin{align*}
			&h'(\tau)\ge -Cr^{-5/2}|\vfd|(B_r(p)),
		\end{align*}
		which, integrating from $s$ to $r$, implies
		\begin{align*}
			&\frac{|\vfd|(B_r(p))}{r^2}-\frac{|\vfd|(B_{(1-C\sqrt r)s}(p))}{s^2}\ge\int_s^r h'(\tau)\,d\tau\ge -Cr^{-3/2}|\vfd|(B_r(p))
		\end{align*}
		and \cref{mono.easy.zero} follows easily.		
		Hence, the density
		\begin{align*}
			&\theta(\vfd,p):=\lim_{s\to 0}\frac{|\vfd|(B_r(p))}{\pi r^2}
		\end{align*}
		exists at any $p\in\subman$. It also follows that
		\begin{align}\label{mono.easy.claim}
			&|\vfd|(B_r(p))\le C(\subman,\bsubman)|\vfd|(\subman)r^2
		\end{align}
		for all $r>0$: this is clear if $p\in\bsubman$, while for $p\nin\bsubman$ and $r\ge\operatorname{dist}(p,\bsubman)$ we have
		$B_r(p)\subseteq B_{2r}(q)$ for some $q\in\bsubman$,
		so that $|\vfd|(B_r(p))\le C|\vfd|(\subman)(2r)^2$, and \cref{mono.easy.claim} follows also for $r<\operatorname{dist}(p,\bsubman)$ thanks to \cref{mono.easy.zero} again. In the same way, using the inclusions $B_s(p)\supseteq B_{s-d}(q)$ and $B_{2d}(q)\supseteq B_d(p)$, with $d:=\operatorname{dist}(p,\bsubman)$ and $q\in\bsubman$ a nearest point to $p$, we deduce that $|\vfd|(B_s(p))\ge cs^2\theta(\vfd,p)$ holds even for $3d<s<\operatorname{diam}(\subman)$.
		Thus,
		\begin{align}\label{mono.easy.claim2}
			&|\vfd|(B_r(p))\ge c(\subman,\bsubman)\theta(\vfd,p)r^2
		\end{align}
		for all $p\in\subman$ and all $0<r<\operatorname{diam}(\subman)$.
	\end{rmk}
	
	\begin{corollary}\label{small.bdry}
		Let $\Phi$ be a $\sigma^5$-critical point for $E_\sigma$, let $\delta>0$,
		and let $U\subseteq\Sigma$ be an open set which intersects $\de\Sigma$ but does not contain entirely any boundary component of $\Sigma$.
		Denote $S_\delta$ the set of points $p\in\subman\setminus\Phi(\de U)$ satisfying the maximal bound
		\begin{align*}
			&\lambda(B_s(p))\le\delta\mu(B_{5s}(p))\quad\text{for all radii }s>0.
		\end{align*}
		Let $T$ be a Borel set of points having distance less than $\sigma$ from $\Phi(\de\Sigma\cap U)$, and such that their distance from $\Phi(\de U)$ is at least $5\sigma$.
		Then we have
		\begin{align*}
			&\mu(S_\delta\cap T)
			\le C\sigma \operatorname{length}(\Phi|_{\de\Sigma\cap U})\frac{\mu(\subman)}{\operatorname{dist}(T,\Phi(\de U))^2}
			+C\sigma^3 \operatorname{length}(\Phi|_{\de\Sigma\cap U}),
		\end{align*}
		for some $C$ depending on $\subman$ and $\bsubman$, provided $\delta\log(1/\sigma)$ is small enough.
	\end{corollary}
	
	\begin{proof}
		Let $L:=\operatorname{length}(\Phi|_{\de\Sigma\cap U})$. We first note that the set of points $T'$ in $\Phi(\de\Sigma\cap U)$ with distance less than $\sigma$ from $T$ can be covered with at most $\sigma^{-1}L$ balls $B_{\sigma}(p_j)$, with $\operatorname{dist}(p_j,\Phi(\de U))\ge 4\sigma$. Indeed, note first that $\operatorname{dist}(T',\Phi(\de U))\ge 4\sigma$; we can discard the components of $\de\Sigma\cap U$ producing an arc of length less than $2\sigma$, since this arc is disjoint from $T'$;
		we are left with finitely many components, corresponding to curves $\gamma_i:I_i\to\bsubman$ with endpoints in $\Phi(\de U)$, where $I_i=(0,|I_i|)$ is an open interval;
		assuming each of them to be parametrized by arclength,
		we then subdivide $[\sigma,|I_i|-\sigma]$ into at most $\sigma^{-1}|I_i|$ intervals $I_{i\ell}$ of size less than $\sigma$
		and we pick a point $p_{i\ell}$ in $\gamma_i(I_{i\ell})\cap T'$, discarding the intervals for which this intersection is empty.
		The resulting collection of balls $\{B_\sigma(p_{i\ell})\}$ is the desired one.
		
		Hence, $T$ is covered by a collection of balls $\{B_{2\sigma}(p_j)\mid j\in J\}$, with $|J|\le \sigma^{-1}L$
		and $\operatorname{dist}(p_j,\Phi(\de U))\ge 4\sigma$.
		
		Now let $J'\subseteq J$ denote the set of indices $j$ such that $B_{2\sigma}(p_j)$ intersects $S_\delta$ and, for $j\in J'$, choose a point $q_j\in S_\delta\cap B_{2\sigma}(p_j)$. Then we have
		\begin{align*}
			&T\cap S_\delta\subseteq\bigcup_{j\in J'}B_{4\sigma}(q_j).
		\end{align*}
		Note that $\operatorname{dist}(q_j,\Phi(\de U))\ge\operatorname{dist}(T,\Phi(\de U))-3\sigma$,
		which is comparable with $\operatorname{dist}(T,\Phi(\de U))$,
		so that \cref{mono.log} gives
		\begin{align*}
			&\frac{\mu(\subman)}{\operatorname{dist}(T,\Phi(\de U))^2}
			\ge (c-C\delta\log(1/\sigma))\frac{\mu(B_{4\sigma}(q_j))}{(4\sigma)^2}-C\sigma^2
		\end{align*}
		for constants $c,C$ depending solely on $\subman,\bsubman$.
		Summing over $j\in J'$, we obtain
		\begin{align*}
			&\mu(T)\le\sum_{j\in J'}\mu(B_{4\sigma}(q_j))
			\le\sigma^{-1}L\frac{C}{1-C\delta\log(1/\sigma)}\Big(\sigma^2\frac{\mu(\subman)}{\operatorname{dist}(T,\Phi(\de U))^2}+\sigma^4\Big)
		\end{align*}
		and the statement follows.
	\end{proof}

	\begin{corollary}\label{mono.log.nobdry}
		Under the same assumptions as in \cref{mono.log}, if $B_\sigma(p)\cap\bsubman=\emptyset$ then
		\begin{align*}
			&\frac{\mu(B_r(p))}{r^2}\ge c-C\delta\log(r/\sigma)-C\sigma^2,
		\end{align*}
		provided $\delta$ and $\sigma$ are small enough.
	\end{corollary}
	
	\begin{proof}
		We first claim that
		\begin{align}\label{sigma.claim}
			&\mu(B_{\sigma}(p))>c'\sigma^2
		\end{align}
		for some universal $c'>0$.
		
		The second fundamental form of the immersed surface $\Phi$ in $\R^\envdim$ is bounded by $|\ff^\Phi|+C(\subman)$, so the monotonicity formula in the ball $\tilde B_t(p):=B_t^{\R^\envdim}(p)$ (see, e.g., \cite[eq.~(17.4)]{simon}, whose proof carries over to the setting of immersed surfaces) and H\"older's inequality give
		\begin{align}\label{easy.mono}\begin{aligned}
			\frac{\mu(\tilde B_t(p))}{t^2}-\frac{\mu(\tilde B_{t/2}(p))}{(t/2)^2}
			&\ge -Ct^{-1}(\sigma^{-1}\lambda(\tilde B_t(p))^{1/4}\mu(\tilde B_t(p))^{3/4}+\mu(\tilde B_t(p))) \\
			&\ge -Ct^{-1}(\sigma^{-1}\delta^{1/4}+1)\mu(\tilde B_{20t}(p))
		\end{aligned}\end{align}
		for $t\le\sigma$ small enough.
		Let $\bar t\le \frac{\sigma}{2}$ be the biggest radius such that $\mu(\tilde B_{\bar t}(p))\ge\frac{\pi}{2}\bar t^2$;
		note that $\bar t$ exists since $\lim_{t\to 0}\frac{\mu(\tilde B_t(p))}{\pi t^2}\ge 1$.
		If $\bar t\ge\frac{\sigma}{80}$ then we are done, thanks to the inclusion $B_{2\bar t}(p)\supseteq\subman\cap \tilde B_{\bar t}(p)$. Otherwise, \cref{easy.mono} gives
		\begin{align*}
			&\frac{\mu(\tilde B_t(p))}{t^2}-\frac{\mu(\tilde B_{t/2}(p))}{(t/2)^2}
			\ge -Ct(\sigma^{-1}\delta^{1/4}+1)
		\end{align*}
		for $\bar t\le t\le\frac{\sigma}{40}$. Setting $t:=2^{-k}(\sigma/40)$ in the last inequality and summing on $k=0,\dots,k_0-1$,
		where $k_0$ is the biggest integer such that $t\ge\bar t$, we get
		\begin{align*}
			&\frac{\mu(\tilde B_{\sigma/40}(p))}{(\sigma/40)^2}\ge \frac{\mu(\tilde B_{\bar t}(p))}{4\bar t^2}-C\delta^{1/4}-C\sigma
		\end{align*}
		and claim \cref{sigma.claim} follows again, for $\delta$ and $\sigma$ small enough.

		The statement now follows by applying \cref{mono.log} with $s:=\sigma$.
	\end{proof}
	
	\section{Asymptotic behavior of the area, in \texorpdfstring{$\Sigma$}{the domain} and in \texorpdfstring{$\subman$}{the ambient}}\label{asympt.sec}
	
	We now investigate the asymptotic behavior of the maps $\Phi_k$ introduced in \cref{crit.sec}.
	Recall that $\nu_k$ is the area measure of $\Phi_k$ on $\Sigma$, meaning that $\nu_k(U)$ is the area of the immersion $\Phi_k|_U$ for any open set $U\subseteq\Sigma$. Also, let $\mu_k:=(\Phi_k)_*\nu_k$ be the corresponding measure on $\subman$, and recall that $\vfd_k$ is the $2$-varifold induced by $\Phi_k$, namely $\vfd_k:=(\Phi_k)_*(\Sigma)$, the varifold pushforward of the canonical multiplicity one $2$-varifold on $\Sigma$.
	
	Up to subsequences, we can assume that $\mu_k$, $\nu_k$ and $\vfd_k$ converge weakly to limits $\mu_\infty$, $\nu_\infty$ and $\vfd_\infty$, in the sense of Radon measures and varifolds.
	
	In this section we show structure theorems for the limit measures $\nu_\infty$, $\mu_\infty$ and for the limit varifold $\vfd_\infty$, namely \cref{ac}, \cref{struct.n} and \cref{par.stat}. The regularity of $\vfd_\infty$ will be studied in \cref{reg.sec}.
	
	We will assume for simplicity that the maps $\Phi_k$ induce the same conformal structure on $\Sigma$; we will discuss the general case later, in \cref{deg.sec}.
	
	Given a reference metric $g_0$ (on $\Sigma$) compatible with this structure, $\operatorname{vol}_{g_0}$ will denote either the corresponding volume form or the associated measure.
	
	Note that $\nu_k=\mz|d\Phi_k|_{g_0}^2\,\operatorname{vol}_{g_0}$. Hence, viewing $\subman\subset\R^\envdim$, the maps $\Phi_k$ are bounded in $W^{1,2}(\Sigma,\R^\envdim)$ and, up to subsequences, we can extract a weak limit $\Phi_\infty$. Note that we have the strong convergence in $L^2$ for the maps $\Phi_k\to\Phi_\infty$ and the traces $\Phi_k|_{\de\Sigma}\to\Phi_\infty|_{\de\Sigma}$; hence, $\Phi_\infty$ and its trace $\Phi_\infty|_{\de\Sigma}$ take values into $\subman$ and $\bsubman$, respectively.
	
	\begin{proposition}\label{quant.phi}
		Given $x\in\Sigma$, fix a local conformal chart centered at $x$ such that
		the chart domain corresponds to $U':=B_1^2$ if $x\nin\de\Sigma$, or to $U':=B_1^2\cap\{\Im(z)\ge 0\}$
		if $x\in\de\Sigma$.
		Given $0<r<1$, assume that $\Phi_k|_{\de B_r^2\cap U'}$ converges to the trace $\Phi_\infty|_{\de B_r^2\cap U'}$ in $C^0$,
		and that $s:=\operatorname{diam}\Phi_\infty(\de B_r^2\cap U')<c_V$, with $c_V$ the constant appearing in \cref{quant.vfd}.
		
		Then either $\limsup_{k\to\infty}\nu_k(B_r^2\cap U')\ge c_Q$, with a constant $c_Q>0$ depending only on $\subman$ and $\bsubman$, or $\spt(\mu)$ is included in a $2s$-neighborhood of $\Phi_\infty(\de B_r^2\cap U')$, for any weak limit $\mu$ of $(\Phi_k|_{B_r^2\cap U'})_*\nu_k$.
	\end{proposition}

	Here the letter $Q$ in $c_Q$ stands for \emph{quantization}; it is not related to the dimension of the Euclidean space $\R^\envdim$.

	\begin{proof}
		Assume $\limsup_{k\to\infty}\nu_k(B_r^2\cap U')<c_Q$, for $c_Q$ to be specified below,
		and let $\mu$ be the weak limit of $(\Phi_k|_{B_r^2\cap U'})_*\nu_k$ along a subsequence (not relabeled).
		The maps $\Phi_k|_{B_r^2\cap U'}$ induce varifolds $\tilde\vfd_k$.
		
		If $x\in\de\Sigma$, then we can repeat the proof of \cref{limit.is.stat} with vector fields $X$ supported outside $\Gamma:=\Phi_\infty(\de B_r^2\cap U')$, with the corresponding variation $w_k$ given by $w_k=X(\Phi_k)$ on $B_r^2\cap U'$ and $w_k=0$ on the complement (in $\Sigma$). We deduce that the limit (up to further subsequences) $\tilde\vfd_\infty$ is a free boundary stationary varifold outside $\Gamma$. If $x\nin\de\Sigma$, then $\tilde\vfd_\infty$ is actually stationary outside $\Gamma$, since any vector field supported outside $\Gamma$ produces a variation which does not change $\Phi_k$ outside $B_r^2$.
		
		Also, if $x\in\de\Sigma$ we let $p_k\in\Phi_k(\de B_r^2\cap\{\Im(z)=0\})\in\bsubman$ and call $p$ any limit point; we then have $p\in\Gamma\cap\bsubman$ and $\Gamma\subseteq \bar B_s(p)$. If $x\nin\de\Sigma$, we just take
		any $p\in\Gamma$ and again we have $\Gamma\subseteq \bar B_s(p)$.
		
		Observing that $(\Phi_k|_{B_r^2\cap U'})_*\nu_k=|\tilde\vfd_k|$ converges both to $\mu$ and to $|\tilde\vfd_\infty|$, we deduce $\mu=|\tilde\vfd_\infty|$.
		Also, $\tilde\vfd_\infty$ has density bounded below by a certain constant $c$, on $\subman\setminus\Gamma$.
		To show this, fix a compact set $K\subset\subman\setminus\Gamma$; it suffices to prove that
		\begin{align}\label{density.claim}
			&\limsup_{k\to\infty}|\tilde\vfd_k|(B_s(q))\ge cs^2
		\end{align}
		for all $s<\operatorname{dist}(K,\Gamma)$ and all $q\in K$ outside a set $F_k$, with $|\tilde\vfd_k|(F_k)\to 0$.
		This can be obtained with \cref{mono.log}, \cref{small.bdry}, \cref{mono.log.nobdry} and a covering argument:
		let
		\begin{align*}
			&\lambda_k:=(\Phi_k|_{B_r^2\cap U'})_*(\sigma_k^4|\ff^{\Phi_k}|^4\,\nu_k),
		\end{align*}
		so that by hypothesis $\lambda_k(\subman)=\frac{\delta_k^2}{\log\sigma_k^{-1}}$
		for some sequence $\delta_k\to 0$.
		Let $F_k'\subseteq K$ be the set of points $q$ such that
		\begin{align*}
			&\lambda_k(B_s(q))>\frac{\delta_k}{\log\sigma_k^{-1}}|\tilde\vfd_k|(B_{5s}(q)),\quad\text{for some }s>0.
		\end{align*}
		Then, by Vitali's covering lemma, we can find a subcollection $\{B_{s_i}(q_i)\}$ of disjoint balls such that
		$F_k'\subseteq\bigcup_i B_{5s_i}(q_i)$.
		This gives $|\tilde\vfd_k|(F_k')\le \frac{\log\sigma_k^{-1}}{\delta_k}\lambda_k(\subman)=\delta_k$, which is infinitesimal.
		
		On the other hand, let $F_k''$ be the $\sigma_k$-neighborhood of $\Phi(\de\Sigma\cap U)$ intersected with $K$.
		Then eventually \cref{small.bdry} is satisfied, with $B_r^2\cap U'$, $F_k''$, $\subman\setminus F_k'$ and $\frac{\delta_k}{\log\sigma_k^{-1}}$ in place of $U$, $T$, $S_\delta$ and $\delta$, and we obtain
		\begin{align*}
			&\mu_k(F_k''\setminus F_k')\le C(K)\sigma_k \operatorname{length}(\Phi_k|_{\de\Sigma}),
		\end{align*}
		which is infinitesimal. Hence, we can set $F_k:=F_k'\cup F_k''$ and, for $q\nin F_k\cup\Gamma$, \cref{mono.log.nobdry} eventually gives
		\begin{align*}
			&\frac{|\tilde\vfd_k|(B_s(q))}{s^2}\ge c-C\frac{\delta_k}{\log\sigma_k^{-1}}\log(s/\sigma_k)-C\sigma_k^2.
		\end{align*}
		The right-hand side converges to $c>0$ as $k\to\infty$, giving \cref{density.claim}.
		
		Hence, if $x\in\de\Sigma$, then $\tilde\vfd_\infty$ satisfies the assumption of \cref{quant.vfd}, and the statement follows. Otherwise, we can conclude using \cref{quant.vfd.bis}.
	\end{proof}

%
%
%

	\begin{thm}\label{ac}
		The limiting measure $\nu_\infty$ has finitely many atoms (possibly none), with weight at least $c_Q$.
		On the complement $\tilde\Sigma$ of this finite set of atoms, $\nu_\infty$ is absolutely continuous with respect to $\operatorname{vol}_{g_0}$ and $\Phi_\infty$ has a continuous representative.
		Moreover, for every open subset $\omega\cptsub\tilde\Sigma$ with $\nu_\infty(\de\omega)=0$, we have $(\Phi_k|_{\omega})_*\nu_k\weakto(\Phi_\infty|_\omega)_*\nu_\infty$.
	\end{thm}

	\begin{proof}
		Given an atom $\{x\}$, we fix a local conformal chart centered at $x$, identifying a neighborhood $U$ of $x$ with the unit disk $U':=B_1^2$ if $x\nin\de\Sigma$, or with $U':=B_1^2\cap\{\Im(z)\ge 0\}$ if $x\in\de\Sigma$.
		
		For all $0<r<1$ we can select $\frac{r}{2}<t<r$ such that $\int_{\de B_{t}^2\cap U'}|d\Phi_\infty|^2\le\frac{2}{r}\int_{B_r^2\cap U'}|d\Phi_\infty|^2$ and such that the trace
		$\Phi_\infty|_{\de B_t^2\cap U'}$ has a $W^{1,2}$ representative, with
		weak derivative given by the restriction of $d\Phi_\infty$ and
		$\Phi_k|_{\de B_t^2\cap U'}\to\Phi_\infty|_{\de B_t^2\cap U'}$ in $C^0$ along a subsequence, which we do not relabel
		(see, e.g., \cite[Lemmas~A.3 and A.5]{pigriv}).
		
		Then, by Cauchy--Schwarz, $s:=\operatorname{diam}(\Phi_\infty(\de B_t^2\cap U'))\le C\Big(\int_{B_r^2\cap U'}|d\Phi_\infty|^2\Big)^{1/2}$
		and hence \cref{quant.phi} is satisfied, if $r$ is small enough.
		Identifying $\nu_k|_U$ with measures on $U'$, we deduce that either $\nu_\infty(\bar B_t^2\cap U')\ge c_Q$ or,
		for some $p\in\subman$,
		\begin{align*}
			\nu_\infty(B_t^2\cap U')
			&\le\liminf_{k\to\infty}\nu_k(B_t^2\cap U')
			=\liminf_{k\to\infty}(\Phi_k|_{B_t^2\cap U'})_*\nu_k(\subman)
			\le\liminf_{k\to\infty}\mu_k(B_{3s}(p)) \\
			&\le\mu_\infty(\bar B_{3s}(p))
			\le Cs^2
			\le C\int_{B_r^2}|d\Phi_\infty|^2.
		\end{align*}
		The penultimate inequality follows from \cref{mono.easy.claim}.
		For $r$ small enough this second possibility cannot happen, since $\nu_\infty(\{0\})>0$.
		Hence we deduce $\nu_\infty(\{x\})\ge c_Q$ and thus there are finitely many atoms.
		
		Assume now that $K$ is a compact set containing no atoms. 
		Assume that $K\subset U$ for a chart domain $U$; we identify $K$ with a compact subset of the unit ball or half unit ball $U'$ as above.
		
		We deal with the half-ball case, whose proof covers also the case $U'=B_1^2$.
		We denote $\de U':=\{z\in U':\Im(z)=0\}$.
		Fix an intermediate set $K\subset V\cptsub U'$ open in $U'$ (hence, $V$ is allowed to contain points in $\de U'$).
		Since $\nu_\infty$ has no atoms on $U$, we can find a radius $r>0$ such that $B_{5r}^2(y)\subseteq B_1^2$, $B_{5r}^2(y)\cap U'\subseteq V$ and $\nu_\infty(B_{5r}^2(y)\cap U')<c_Q$, for all $y\in K$.
		
		Taking a maximal subset of centers $\{y_i'\}\subseteq K$ with pairwise distances at least $\frac{r}{2}$, we can cover $K$ with a finite collection of balls $\{B_{r/2}^2(y_i')\}$ with $\sum_i \uno_{B_{5r}^2(y_i')}\le C$. If $B_r^2(y_i')\subseteq U'$ then we set $y_i:=y_i'$ and $r_i:=r$; otherwise we choose $y_i$ to be a point in $B_r^2(y_i')\cap\de U'$, and we set $r_i:=4r$. Note that $B_{r/2}^2(y_i')\subseteq B_{r_i/2}^2(y_i)$ and $B_{r_i}^2(y_i)\subseteq B_{5r}^2(y_i')$, so the collection of balls $B_{r_i/2}^2(y_i)$ still covers $K$ and has $\sum_i\uno_{B_{r_i}^2(y_i)}\le C$.
		Moreover, either $B_{r_i}^2(y_i)\subseteq V$ or $y_i\in\de U'$, with $B_{r_i}^2(y_i)\cap U'\subseteq V$.
		Also, $\nu_\infty(B_{r_i}^2(y_i)\cap U')<c_Q$.
		
		
		We can fix $t_i\in(\frac{r_i}{2},r_i)$ such that $\Phi_k|_{\de B_{t_i}^2(y_i)\cap U'}\to\Phi_\infty|_{\de B_{t_i}^2(y_i)\cap U'}$ in $C^0$ along a subsequence independent of $i$, and such that the diameter $s_i$ of $\Phi_\infty(\de B_{t_i}^2(y_i)\cap U')$ satisfies
		\begin{align*}
		&s_i^2\le C\int_{B_{r_i}^2(y_i)\cap U'}|d\Phi_\infty|^2.
		\end{align*}
		We now work along this subsequence, which we do not relabel.
		By \cref{quant.phi}, if $r$ was chosen small enough, any weak limit of the measures $(\Phi_k|_{B_{t_i}^2(y_i)\cap U'})_*\nu_k$ is supported in $B_{3s_i}(p_i)$ for some $p_i\in\subman$. Hence,
		\begin{align}\label{vanish.target}
			&\lim_{k\to\infty}(\Phi_k|_{B_{t_i}^2(y_i)\cap U'})_*\nu_k(\subman\setminus B_{3s_i}(p_i))=0.
		\end{align}
		Since $\nu_k=\mz|d\Phi_k|^2\,\mathcal{L}^2$ on $U'$, setting $h_i:=(\operatorname{dist}(\cdot,p_i)-3s_i)^+$ we deduce
		that $h_i\circ\Phi_k\to 0$ in $W^{1,2}(B_{t_i}^2(y_i)\cap U')$. Hence, the essential image of $\Phi_\infty|_{B_{t_i}^2(y_i)\cap U'}$ is included in $\bar B_{3s_i}(p_i)$. We deduce
		\begin{align*}
			\int_K\operatorname{dist}(\Phi_k,\Phi_\infty)\,d\nu_k
			&\le\sum_i\int_{B_{t_i}^2(y_i)\cap U'}\operatorname{dist}(\Phi_k,\Phi_\infty)\,d\nu_k \\
			&\le\sum_i 6s_i\nu_k(B_{t_i}^2(y_i)\cap U') \\
			&\quad+\operatorname{diam}(\subman)\sum_i(\Phi_k|_{B_{t_i}^2(y_i)\cap U'})_*\nu_k(\subman\setminus B_{3s_i}(p_i)) \\
			&\le C(\sup s_i)\nu_k(V)+C\sum_i(\Phi_k|_{B_{t_i}^2(y_i)\cap U'})_*\nu_k(\subman\setminus B_{3s_i}(p_i)).
		\end{align*}
		In the limit $k\to\infty$, using \cref{vanish.target}, we get
		\begin{align*}
			&\limsup_{k\to\infty}\int_K\operatorname{dist}(\Phi_k,\Phi_\infty)\,d\nu_k
			\le C(\sup_i s_i)\nu_\infty(\bar V).
		\end{align*}
		Since we could arrange that $\sup_i s_i$ is arbitrarily small, we arrive at
		\begin{align}\label{dist.to.zero}
			&\lim_{k\to\infty}\int_K\operatorname{dist}(\Phi_k,\Phi_\infty)\,d\nu_k=0.
		\end{align}
		Also, choosing $\eta$ so small that any ball $B_\eta^2(y)$ is included in some $B_{t_i}^2(y_i)$, for all $y\in K$,
		the essential oscillation of $\Phi_\infty|_{B_\eta^2(y)\cap U'}$ is then bounded by $\sup_i s_i$. Since the latter is arbitrarily small, it follows that $\Phi_\infty$ has a continuous representative on $K$, hence on $\tilde\Sigma$.
		
		Finally, if $\mathcal{L}^2(K)=0$ then, arguing as in the first part of the proof, we have
		\begin{align*}
			\nu_\infty(K)
			&\le\sum_i\nu_\infty(B_{t_i}^2(y_i)\cap U')
			\le\liminf_{k\to\infty}\sum_i\nu_k(B_{t_i}^2(y_i)\cap U')
			\le C\sum_i s_i^2 \\
			&\le C\sum_i\int_{B_{r_i}^2(y_i)\cap U'}|d\Phi_\infty|^2
			\le C\int_V|d\Phi_\infty|^2.
		\end{align*}
		Since $V$ is an arbitrary neighborhood of $K$, we deduce $\nu_\infty(K)\le C\int_K|d\Phi_\infty|^2=0$.
		The absolute continuity of $\nu_\infty$ with respect to $\operatorname{vol}_{g_0}$ on $\tilde\Sigma$ follows.
		
		Finally, given $\omega$ as in the statement and covering $\bar\omega$ with finitely many charts, it follows from \cref{dist.to.zero} that $\lim_{k\to\infty}\int_{\bar\omega}\operatorname{dist}(\Phi_k,\Phi_\infty)\,d\nu_k=0$. Hence, for any $\psi\in C^0(\subman)$,
		\begin{align*}
			&\lim_{k\to\infty}\int_\omega\psi\circ\Phi_k\,d\nu_k
			=\lim_{k\to\infty}\int_\omega\psi\circ\Phi_\infty\,d\nu_k
			=\int_\omega\psi\circ\Phi_\infty\,d\nu_\infty
		\end{align*}
		the last equality coming from the continuity of $\psi\circ\Phi_\infty$ near $\bar\omega$ and the assumption $\nu_\infty(\de\omega)=0$. The weak convergence $(\Phi_k|_\omega)_*\nu_k\weakto(\Phi_\infty|_\omega)_*\nu_\infty$ follows.
	\end{proof}

	\begin{thm}\label{struct.n}
		The absolutely continuous part of $\nu_\infty$, which we denote $m\,\operatorname{vol}_{g_0}$,
		has $m=0$ a.e.\ on the set of points where $d\Phi_\infty$ does not have rank $2$.
		Moreover, $m=N J(d\Phi_\infty)$ for a bounded, integer valued function $N\ge 1$.
	\end{thm}

	In the statement $J(d\Phi_\infty)$ denotes the Jacobian of $\Phi_\infty$ with respect to the volume form $\operatorname{vol}_{g_0}$.
	Hence, in a conformal chart, we are asserting that the absolutely continuous part of $\nu_\infty$
	is $N|\de_1\Phi_\infty\wedge\de_2\Phi_\infty|\,\mathcal{L}^2$.
	
	\begin{proof}
		Working in a conformal chart for $\operatorname{int}(\Sigma)$, we fix a point $x$ which is Lebesgue for $d\Phi_\infty$, and such that $\nu_\infty(\{x\})=0$.
		We have to show that $\frac{\nu_\infty(B_r^2(x))}{\pi r^2}\to N|\de_1\Phi_\infty\wedge\de_2\Phi_\infty|(x)$ for some bounded integer $N\ge 1$, as $r\to 0$ along some sequence.
		
		We can assume $x=0$.
		For all $r>0$ small enough, call $\vfd_{k,r}$ the varifold induced by $\Phi_k|_{B_r^2}$.
		We can select an arbitrarily small $r$ such that the trace $\Phi_\infty|_{\de B_r^2}$ has
		\begin{align*}
			&\Phi_\infty(ry)=\Phi_\infty(0)+rd\Phi_\infty(0)[y]+o(r) \quad \text{for }|y|=1
		\end{align*}
		and such that the traces $\Phi_k|_{\de B_r^2}$ converge subsequentially to $\Phi_\infty|_{\de B_r^2}$ in $C^0$ (see, e.g., \cite[Lemmas~A.4 and A.5]{pigriv}).
		By \cref{quant.phi}, any (subsequential) weak limit of $|\vfd_{k,r}|$ is supported in a ball $B_{Cr}(p)$, with $p:=\Phi_\infty(0)$ and $C$ depending also on $|d\Phi_\infty(0)|$.
		
		Moreover, any (subsequential) limit $\vfd=\lim_{k\to\infty}\vfd_{k,r}$ is stationary in $\subman\setminus\Phi_\infty(\de B_r^2)$
		and satisfies $|\vfd|(B_s(q))\le Cs^2$ for all $q\in\subman$,
		since the varifolds $\vfd_k$ induced by $\Phi_k$ (from the full domain) have trivially $|\vfd_k|\ge|\vfd_{k,r}|$
		and, by \cref{limit.is.stat}, they converge subsequentially to a free boundary stationary varifold $\vfd_\infty$, for which \cref{mono.easy.claim} gives the desired bound.
		
		Hence, with a diagonal argument, we may find a subsequence of $k$'s (not relabeled) and a sequence of radii $r_k\to 0$ such that the dilated varifolds
		$\vfd_k':=(r_k^{-1}(\cdot-p))_*\vfd_{k,r_k}$ in $\R^\envdim$ form a tight sequence, converging to a varifold $\vfd_\infty'$ which has
		\begin{align}\label{mono.n}
			&|\vfd_\infty'|(B_s^\envdim(q))\le Cs^2\quad\text{for all }q\in\R^\envdim\text{ and all }s>0
		\end{align}
		with a constant $C$ independent of $x$, has compact support and is stationary in $\R^\envdim\setminus\mathcal{C}$, with
		\begin{align*}
			&\mathcal{C}=\lim_{k\to\infty}(r_k^{-1}\Phi_\infty(\de B_{r_k}^2)-p)=\{d\Phi_\infty(0)[y]\mid y\in\de B_1^2\}.
		\end{align*}
		We can also assume that
		\begin{align}\label{slow.radii}
			&r_k^{-2}\sigma_k^4\int_{B_{r_k}^2}f_k^2\,d\nu_k\to 0,\quad r_k^{-1}\sigma_k\to 0,
		\end{align}
		and that
		\begin{align*}
			&|\vfd_\infty'|(\R^\envdim)
			=\lim_{k\to\infty}|\vfd_k'|(\R^\envdim)
			=\lim_{k\to\infty}\frac{\nu_\infty(B_{r_k}^2)}{r_k^2}
			=\lim_{k\to\infty}\frac{\nu_k(B_{r_k}^2)}{r_k^2};
		\end{align*}
		since the convex hull $\operatorname{co}(\mathcal{C})$ of $\mathcal{C}$ has area $\pi|\de_1\Phi_\infty\wedge\de_2\Phi_\infty|(0)$, we are left to show that
		\begin{align*}
			&|\vfd_\infty'|(\R^\envdim)
			=N\mathcal{H}^2(\operatorname{co}(\mathcal{C}))
		\end{align*}
		for some bounded integer $N\ge 1$.
		By \cite[Theorem~19.2]{simon}, which holds for general varifolds, $|\vfd_\infty'|$ is supported in the convex hull of $\mathcal{C}$.
		If $d\Phi_\infty(0)$ has rank less than $2$, then $\mathcal{C}$ is either a segment or a point.
		Hence, we can cover it with $O(s^{-1})$ balls of radius $s$; recalling \cref{mono.n}, we deduce $|\vfd_\infty'|(\mathcal{C})=0$ and hence $\vfd_\infty'=0$. Thus the claim follows in this case.

		
		If instead $d\Phi_\infty(0)$ has rank $2$, we first observe that the area of the map $\Psi_k:=r_k^{-1}(\Phi_k|_{B_{r_k}^2}-p)$
		is, up to an infinitesimal error, at least the area of $\operatorname{co}(\mathcal{C})$ in the plane $\Pi$ containing it:
		this follows immediately considering the composition $\bar\Psi_k$ of this map with the projection onto $\Pi$, and noting that any compact subset $K\subset\operatorname{co}(\mathcal{C})\setminus\mathcal{C}$ belongs eventually to the image of $\bar\Psi_k$, since $\bar\Psi_k$ has (eventually) nontrivial degree relative to the points in $K$. Hence,
		\begin{align}\label{lower.bd.degree}
			&\mathcal{H}^2(\operatorname{co}(\mathcal{C}))
			\le\lim_{k\to\infty}\frac{\nu_k(B_{r_k}^2)}{r_k^2}
			=|\vfd_\infty'|(\R^\envdim).
		\end{align}
		Up to rotations, we can assume $\Pi=\R^2\times\{0\}$.
		Since $\mathcal{C}$ is a smooth curve, we have $|\vfd_\infty'|(\mathcal{C})=0$.
		Also, $\vfd_\infty'$ is stationary on $\R^\envdim\setminus\mathcal{C}$ and supported on $\Pi$.
		By the constancy theorem \cite[Theorem~41.1]{simon}, it follows that $\vfd_\infty'$ is rectifiable
		and equals a multiple $N$ of $\operatorname{co}(\mathcal{C})$. By \cref{lower.bd.degree} we have $N\ge 1$,
		while from \cref{mono.n} it follows that $N\le C$. We are left to show $N\in\N$.
		
		Note that $\vfd_k'$ is the varifold induced by $\Psi_k$; hence, the varifold convergence $\vfd_k'\weakto\vfd_\infty'$ implies that
		\begin{align}\label{small.vert}
			&\int_{B_{r_k}^2}|d\Psi_k^j|^2\to 0\quad\text{for }j=3,\dots,\envdim,
		\end{align}
		where we write $\Psi_k=(\Psi_k^1,\dots,\Psi_k^\envdim)$.
		
		Fix $\alpha>0$ such that $\mathcal{C}$ encloses a ball $B_{2\alpha}^2$ in the plane $\Pi$.
		Consider a family $(\rho_\tau)$ of mollifiers in $\R^\envdim$, namely nonnegative smooth functions supported in $B_\tau^\envdim$ with $\int_{\R^\envdim}\rho_\tau=1$ and $|d\rho_\tau|\le C\tau^{-\envdim-1}$.
		For any vector field $X\in C^\infty_c(B_\alpha^2,\R^2)$, viewing $X$ as a vector field on $\R^\envdim$, constant in the last $\envdim-2$ variables, we define the vector fields $X_k$ and $Y_k$ on $\subman_k:=r_k^{-1}(\subman-p)$ given pointwise by the projection of $X$ and $\rho_{\tau_k}*X$ onto the tangent space to $\subman_k$, respectively, with $\tau_k:=r_k^{-1}\sigma_k$.
		
		Since $\subman_k$ converges to an $m$-plane graphically (in any neighborhood of $0$), we have
		\begin{align}\label{close.nabla}
			&|\nabla^{\subman_k}_vX_k-\nabla^{\R^\envdim}_vX|\le\delta_k\|dX\|_{L^\infty}|v|
		\end{align}
		for some sequence $\delta_k\to 0$ and any $v\in T\subman_k$. Also, we have
		\begin{align*}
			&|(\nabla^{\subman_k})^2 Y_k|
			\le C\|\rho_{\tau_k}*X\|_{C^2}
			\le C\tau_k^{-1}\|dX\|_{L^\infty}.
		\end{align*}
		
		Note that $\Psi_k$, when extended to $\Sigma$ with the same formula $r_k^{-1}(\Phi_k-p)$, is $\tau_k^5$-critical for $E_{\tau_k}$ for the manifold $\subman_k$ and the corresponding Finsler manifold $\bman_k$:
		indeed, identifying $T_{\Phi_k}\bman$ and $T_{\Psi_k}\bman_k$ with subsets of $W^{2,4}(\Sigma,\R^\envdim)$, for all $w\in T_{\Psi_k}\bman_k$ we have
		\begin{align*}
			|dE_{\tau_k}(\Psi_k)[w]|
			&=r_k^{-2}|dE_{\sigma_k}(\Phi_k)[r_k w]|
			\le r_k^{-1}\sigma_k^5\|w\|_{\Phi_k}
			\le r_k^2\tau_k^5\|w\|_{\Phi_k}
		\end{align*}
		and it is immediate to check that $\|w\|_{\Phi_k}\le r_k^{-3/2}\|w\|_{\Psi_k}\le r_k^{-2}\|w\|_{\Psi_k}$
		(assuming $r_k\le 1$).
		
		For the vector field $Y_k$, recalling \cref{R.bound}, the term $f\ang{\nabla n,\nabla\omega}$ in \cref{first.var.X}
		is bounded by
		\begin{align*}
			&C|\ff^{\Psi_k}|^4\|\nabla^{\subman_k}Y_k\|_{L^\infty}
			+C|\ff^{\Psi_k}|^3\|(\nabla^{\subman_k})^2 Y_k\|_{L^\infty} \\
			&\le C(|\ff^{\Psi_k}|^4+\tau_k^{-1}|\ff^{\Psi_k}|^3)\|dX\|_{L^\infty}.
		\end{align*}
		We now use the almost criticality of $\Psi_k$ with the infinitesimal variation $Y_k(\Psi_k)$, or more precisely
		$Y_k(\Psi_k)\uno_{B_{r_k}^2}$ for the extension $r_k^{-1}(\Phi_k-p)\in\bman_k$ of $\Psi_k$.
		For $k$ large enough, $\Psi_k(\de B_{r_k}^2)$ does not intersect $B_{2\alpha}^2\times\R^{\envdim-2}$,
		where $Y_k$ is supported, and hence this is an admissible variation.
		As in the proof of \cref{limit.is.stat}, since the immersions $\Psi_k$ have bounded area we obtain
		\begin{align*}
			\|Y_k(\Psi_k)\|_{\Psi_k}
			&\le C\|dX\|_{L^\infty}+C\|(\nabla^{\subman_k})^2 Y_k\|_{L^\infty}+\|dX\|\Big(\int_{B_{r_k}^2}|\ff^{\Psi_k}|^4\,\operatorname{vol}_{\Psi_k}\Big)^{1/4} \\
			&\le C\tau_k^{-1}\|dX\|_{L^\infty}\Big(1+\Big(\int_{B_{r_k}^2}\tau_k^4|\ff^{\Psi_k}|^4\,\operatorname{vol}_{\Psi_k}\Big)^{1/4}\Big)
		\end{align*}
		for some $C$ independent of $X$.
		Hence, \cref{first.var.X} and the $\tau_k^5$-criticality of $\Psi_k$, together with Young's inequality, give
		\begin{align*}
			\Big|\int_{B_{r_k}^2}\ang{\de_i\Psi_k, \nabla Y_k(\Psi_k)[\de_i\Psi_k]}\Big|
			&\le C\|dX\|_{L^\infty}\int_{B_{r_k}^2}(\tau_k^4|\ff^{\Psi_k}|^4+\tau_k^3|\ff^{\Psi_k}|^3)\,\operatorname{vol}_{\Psi_k} \\
			&\quad +C\tau_k^4\|dX\|_{L^\infty}.
		\end{align*}
		Since $\tau_k^4|\ff^{\Psi_k}|^4=\sigma_k^4|\ff^{\Phi}_k|^4$, H\"older's inequality gives the upper bound
		\begin{align*}
			&C\Big(r_k^{-2}\sigma_k^4\int_{B_{r_k}^2}|\ff^{\Phi_k}|^4\,\operatorname{vol}_{\Phi_k}
			+\Big(r_k^{-2}\sigma_k^4\int_{B_{r_k}^2}|\ff^{\Phi_k}|^4\,\operatorname{vol}_{\Phi_k}\Big)^{3/4}
			+\tau_k^4\Big)\|dX\|_{L^\infty}
		\end{align*}
		for the previous right-hand side. In view of \cref{slow.radii}, it follows that
		\begin{align*}
			&\Big|\int_{B_{r_k}^2}\ang{\de_i\Psi_k, \nabla Y_k(\Psi_k)[\de_i\Psi_k]}\Big|
			\le\delta_k'\|dX\|_{L^\infty}
		\end{align*}
		for some $\delta_k'\to 0$ independent of $X$.
		Also, replacing $Y_k$ with $X_k$ is not harmful, since the last integral is the first variation of the area and thus
		\begin{align*}
			&\Big|\int_{B_{r_k}^2}\ang{\de_i\Psi_k, \nabla Y_k(\Psi_k)[\de_i\Psi_k]}-\int_{B_{r_k}^2}\ang{\de_i\Psi_k, \nabla X_k(\Psi_k)[\de_i\Psi_k]}\Big| \\
			&\le 2\int_{B_{r_k}^2}|H^{\Psi_k}||\rho_{\tau_k}*X-X|\,\operatorname{vol}_{\Psi_k} \\
			&\le C\tau_k\|dX\|_{L^\infty}\int_{B_{r_k}^2}|\ff^{\Psi_k}|\,\operatorname{vol}_{\Psi_k}
		\end{align*}
		is infinitesimal with respect to $\|dX\|_{L^\infty}$.
		Choose now $X:=\varphi(x_1,x_2)e_1$.
		Writing $\Psi_k=(\Psi_k^1,\dots,\Psi_k^m)$, in view of \cref{close.nabla} the previous integral (with $X_k$ replacing $Y_k$) is just
		\begin{align*}
			&\int_{B_{r_k}^2}(\de_1\varphi(\Psi_k) \de_i\Psi_k^1 \de_i\Psi_k^1
			+\de_2\varphi(\Psi_k) \de_i\Psi_k^1 \de_i\Psi_k^2),
		\end{align*}
		up to another infinitesimal error.
		Let $J_k:=|\de_1\Psi_k^1\de_2\Psi_k^2-\de_2\Psi_k^1\de_1\Psi_k^2|$ denote the Jacobian of the composition of $\Psi$ with the projection onto $\Pi$.
		Using \cref{small.vert}, this integral equals
		\begin{align*}
			&\int_{B_{r_k}^2}J_k\de_1\varphi(\Psi_k)\,\operatorname{vol}_{\Psi_k}
		\end{align*}
		plus an error which is infinitesimal with respect to $\|d\varphi\|_{L^\infty}$		
		(see also \cite[Lemma~A.6]{pigriv}).
		Hence, by the area formula, the projection $\vfd_k''$ of $\vfd_k'$ onto $\Pi$ has an integer multiplicity $N_k$
		satisfying
		\begin{align*}
			&\Big|\int_{B_\alpha^2}N_k\de_1\varphi\,d\mathcal{L}^2\Big|\le\delta_k''\|dX\|_{L^\infty}\quad\text{with }\delta_k''\to 0
		\end{align*}
		and, using the vector field $\varphi(x_1,x_2)e_2$, the same holds replacing $\de_1\varphi$ with $\de_2\varphi$.
		So, by Allard's strong constancy lemma \cite[Theorem~1.(4)]{allard.const},
		it follows that $N_k$ is close in $L^1$ to a constant $\bar N_k$ on the ball $B_{\alpha/2}^2$, with a distance $O(\delta_k'')$. As $N_k$ is integer valued, it follows that $\operatorname{dist}(\bar N_k,\N)\to 0$.
		Finally, since $\vfd_k''$ converges to $\vfd_\infty'$,
		we have
		\begin{align*}
			&\pi(\alpha/2)^2 N=\lim_{k\to\infty}\int_{B_{\alpha/2}^2}N_k\,d\mathcal{L}^2
			=\lim_{k\to\infty}\pi(\alpha/2)^2\bar N_k
		\end{align*}
		and we deduce $N\in\N$.
	\end{proof}
	
	\begin{rmk}
		Note that, in the previous proof, testing immediately the stationarity of $\Psi_k$ against $X_k$ would have run into trouble,
		since we would have got a bound for $\int_{B_\alpha^2}N_k\de_i\varphi$ depending also on the Hessian of $X$, making it impossible to apply Allard's strong constancy lemma.
	\end{rmk}

	\begin{definition}\label{v.omega}
		Given an open set $\omega\subseteq\tilde\Sigma$,
		we define the subset $\mathcal{G}_\omega\subseteq\omega\setminus\de\Sigma$ of Lebesgue points for $d\Phi_\infty$ where this differential has rank $2$.
		We equip the image $\Phi_\infty(\mathcal{G}_\omega)$ with the multiplicity
		\begin{align*}
			&\theta_\omega(p):=\sum_{x\in\mathcal{G}_\omega\cap\Phi_\infty^{-1}(p)}N(x).
		\end{align*}
		Note that, by the area formula (see, e.g., \cite[Lemma~A.2]{pigriv}),
		$\Phi_\infty(\mathcal{G}_\omega)$ is $2$-rectifiable
		and $\theta_\omega$ is well defined as a function in $L^1(\mathcal{H}^2\mrestr\Phi_\infty(\mathcal{G}_\omega))$. We can then view this set, with multiplicity $\theta_\omega$, as a rectifiable varifold in $\subman$, which we call $\vfd_\omega$.
	\end{definition}
	
	Note that, by \cref{ac}, \cref{struct.n} and the area formula, the weight $|\vfd_\omega|$
	coincides with $(\Phi_\infty|_\omega)_*\nu_\infty$.
	
	\begin{proposition}\label{var.conv}
		Given an open subset $\omega\cptsub\tilde\Sigma$ with $\nu_\infty(\de\omega)=0$, the immersions $\Phi_k|_\omega$ converge to the varifold $\vfd_\omega$.
	\end{proposition}
	
	\begin{proof}
		By splitting $\omega$ into finitely many pieces with $\nu_\infty$-negligible boundary, we can reduce to the case that $\omega$ is contained in a local chart; in the sequel, we identify $\omega$ with a subset of $\C=\R^2$.
	
		Let $\vfd_{k,\omega}:=(\Phi_k)_*(\omega)$ be the varifold induced by $\Phi_k|_\omega$. Viewing $\vfd_{k,\omega}$ (for $k\le\infty$) as a varifold in $\R^\envdim$, by the area formula and \cref{struct.n} it suffices to show that
		\begin{align}\label{var.conv.claim0}
			&\int_\omega \varphi(\Phi_k(x),d\Phi_k(x)[T_x\Sigma])\,d\nu_k(x)
			\to\int_\omega\varphi(\Phi_\infty(x),d\Phi_\infty(x)[T_x\Sigma])\,d\nu_\infty(x)
		\end{align}
		for any $\varphi\in C^1_c(\operatorname{Gr}_2(\R^\envdim))$.
		The last integrand is meant to be zero at points where $d\Phi_\infty$ does not have full rank.
		In order to simplify the notation, we indicate
		the plane $d\Phi_k(x)[T_x\Sigma]=d\Phi_k(x)[\R^2]$ by $\Pi_k(x)$.
		
		Let $\mathcal{G}_\omega'$ be the subset of $\mathcal{G}_\omega$ consisting of the points $x$ where additionally $\int_{B_r^2(x)}|d\Phi_\infty-d\Phi_\infty(x)|^2=o(r^2)$.
		For any point $x\in\mathcal{G}_\omega'$,
		pick a sequence of radii $r$ satisfying
		\begin{align*}
			&|\Phi_\infty(x+ry)-\Phi_\infty(x)-d\Phi_\infty(x)[ry]|=o(r)\quad\text{for }|y|=1.
		\end{align*}
		Given any $\epsilon>0$, we claim that
		\begin{align}\label{var.conv.claim}
			&\limsup_{k\to\infty}\int_{B_r^2(x)}|\varphi(\Phi_k,\Pi_k)-\varphi(\Phi_\infty(x),\Pi_\infty(x))|\,d\nu_k\le\epsilon r^2
		\end{align}
		for $r$ small enough in this sequence.
		If this is not true, then using a diagonal argument as in the proof of \cref{struct.n} we may find a subsequence of $k$'s (not relabeled) and radii $r_k\to 0$
		such that
		\begin{align}\label{var.conv.zero}
			&\int_{B_{r_k}^2(x)}|\varphi(\Phi_k,\Pi_k)-\varphi(\Phi_\infty(x),\Pi_\infty(x))|\,d\nu_k\ge\epsilon r_k^2
		\end{align}
		as well as
		\begin{align*}
			&r_k^{-2}|\nu_k(B_{r_k}^2(x))-\nu_\infty(B_{r_k}^2(x))|\to 0,
		\end{align*}
		and such that the varifolds induced by $\Psi_k:=r_k^{-1}(\Phi_k|_{B_{r_k}^2(x)}-\Phi_\infty(x))$
		converge tightly to a rectifiable varifold $\vfd'$ supported in a bounded subset of $\Pi_\infty(x)$.
		In particular, the $\nu_k$-measure of the set of points in $B_{r_k}^2(x)$ where $|\Psi_k|>r_k^{-1/2}$ is $o(r_k^2)$.
		Since $\Phi_k=\Phi_\infty(x)+r_k\Psi_k$, we deduce that
		\begin{align}\label{var.conv.on}
			&\int_{B_{r_k}^2(x)}|\varphi(\Phi_k,\Pi_k)-\varphi(\Phi_\infty(x),\Pi_k)|\,d\nu_k
			\le r_k^{1/2}\|d\varphi\|_{L^\infty}\nu_k(B_{r_k}^2(x))+o(r_k^2)
			=o(r_k^2)
		\end{align}
		as $k\to\infty$. Also, testing the tight varifold convergence of $\Psi_k$ to $\vfd'$ against the function $|\varphi(\Phi_\infty(x),\cdot)-\varphi(\Phi_\infty(x),\Pi_\infty(x))|$, we get
		\begin{align}\label{var.conv.dos}
			&\int_{B_{r_k}^2(x)}|\varphi(\Phi_\infty(x),\Pi_k)-\varphi(\Phi_\infty(x),\Pi_\infty(x))|\,d\nu_k=o(r_k^2).
		\end{align}
		Combining \cref{var.conv.on} with \cref{var.conv.dos} we get a contradiction to \cref{var.conv.zero}.
		By the Besicovitch covering lemma, we can then cover any fixed compact set $K\subseteq\mathcal{G}_\omega'$ with finitely many balls $\{B_j\}$ included in $\omega$ such that \cref{var.conv.claim} holds, for $B_j=B_{r_j}^2(x_j)$ in place of $B_r^2(x)$, as well as
		\begin{align*}
			&\int_{B_j}|\varphi(\Phi_\infty,\Pi_\infty)-\varphi(\Phi_\infty(x_j),\Pi_\infty(x_j))|\,d\nu_\infty\le\epsilon r_j^2,
		\end{align*}
		using the approximate continuity of $d\Phi_\infty$ at points in $K$,
		and such that $\sum_j\uno_{B_j}\le C$.
		
		Let $\mathcal{B}_j:=B_j\setminus\bigcup_{\ell<j}B_\ell$ and $U:=\bigcup_j B_j=\bigcup_j\mathcal{B}_j\subseteq\omega$.
		Since $\sum_j r_j^2\le C$, we deduce that
		\begin{align*}
			&\limsup_{k\to\infty}\Big|\int_U\varphi(\Phi_k,\Pi_k)\,d\nu_k
			-\int_U\varphi(\Phi_\infty,\Pi_\infty)\,d\nu_\infty\Big| \\
			&\le\sum_j\limsup_{k\to\infty}\Big|\int_{\mathcal{B}_j}\varphi(\Phi_\infty(x_j),\Pi_\infty(x_j))\,d\nu_k
			-\int_{\mathcal{B}_j}\varphi(\Phi_\infty(x_j),\Pi_\infty(x_j))\,d\nu_\infty\Big|+C\epsilon.
		\end{align*}
		The sum vanishes, since $\nu_\infty(\de\mathcal{B}_j)=0$.
		Also, since $\nu_\infty(\de\omega)=0$, we have
		\begin{align*}
			&\limsup_{k\to\infty}\nu_k(\omega\setminus U) \le\nu_\infty(\omega\setminus U) \le\nu_\infty(\omega\setminus K)
		\end{align*}
		and this quantity
		can be made arbitrarily small, proving \cref{var.conv.claim0}.
	\end{proof}
	
	\begin{definition}
		We say that a property holds \emph{for a.e.}\ $\omega\subseteq\Sigma$ if, for every nonnegative $\rho\in C^\infty(\Sigma)$, it holds for a.e.\ superlevel set $\{\rho>\lambda\}$ with $\lambda>0$.
		Similarly, we say that it holds \emph{for a.e.}\ $\omega\cptsub\operatorname{int}(\Sigma)$ if we have the same for every nonnegative $\rho\in C^\infty_c(\operatorname{int}(\Sigma))$.
	\end{definition}
	
	\begin{definition}
		A map $\Phi\in W^{1,2}(\Sigma)$ is \emph{weakly conformal}
		if, for a.e.\ $x\in\Sigma$, its differential at $x$ is zero or a linear conformal map $T_x\Sigma\to T_{\Phi(x)}\subman$.
	\end{definition}
	
	\begin{definition}\label{par.def}
		Let $\Sigma$ be a compact Riemann surface, possibly with boundary, $\Phi\in W^{1,2}(\Sigma,\subman)$ weakly conformal with $\Phi(\de\Sigma)\subseteq\bsubman$,
		and $N\in L^\infty(\Sigma,\{1,2,\dots\})$.
		The triple $(\Sigma,\Phi,N)$ is a \emph{parametrized free boundary stationary varifold} if, for almost every $\omega\subseteq\Sigma$, the varifold $\Phi_*(N\omega)$ is free boundary stationary (for $\subman,\bsubman$) outside $\Phi_\infty(\de\omega)$ (see \cref{fb.def})
		and if, for almost every $\omega\cptsub\operatorname{int}(\Sigma)$, $\Phi_*(N\omega)$ is stationary outside $\Phi_\infty(\de\omega)$.
	\end{definition}
	
	The pushforward $\Phi_*(N\omega)$ in this definition has to be interpreted as the varifold $\vfd_\omega$ in \cref{v.omega},
	using the subset of $\omega$ made of Lebesgue points, for $\Phi$ and $d\Phi$, where $d\Phi$ has rank $2$.
	
	\begin{rmk}
		In this definition, $N\omega$ is viewed as a $2$-varifold in the surface $\Sigma$, equipped with a metric compatible with the conformal structure; however, $\Phi_*(N\omega)$ is independent of the choice of the metric.
		Note that $\de\omega$ is a compact one-dimensional submanifold and the trace $\Phi|_{\de\omega}$ has a continuous representative for a.e.\ $\omega$: this follows, e.g., from \cite[Theorems~4.21 and 5.7]{evans} applied to the regular level sets of $f$; $\Phi(\de\omega)$ implicitly refers to the (compact) image by means of this continuous map. Note also that the definition does not depend on the representatives of $\Phi$ and $N$.
	\end{rmk}

	\begin{thm}\label{par.stat}
		There exists a compact Riemann surface $\Sigma'$ and a quasiconformal homeomorphism $\varphi:\Sigma'\to\Sigma$
		such that $(\Sigma',\Phi_\infty\circ\varphi,N\circ\varphi)$ is a free boundary parametrized stationary varifold for $(\subman,\bsubman)$.
	\end{thm}
	
	We refer the reader to \cite[Chapter~4]{imayoshi} for basic properties of quasiconformal homeomorphisms.
	
	\begin{proof}
		For a.e.\ $\omega\subseteq\Sigma$, $\Phi_k|_{\de\omega}$ converges in $C^0$ to $\Phi_\infty|_{\de\omega}$ (up to subsequences) and $\de\omega\cap\mathcal{A}=\emptyset$, with $\mathcal{A}$ the finite set of atoms for $\nu_\infty$.
		
		With respect to the fixed metric $g_0$ on $\Sigma$,
		we can find an arbitrarily small radius $r>0$ such that
		for any $x\in\omega\cap\mathcal{A}$ we have $B_r(x)\cptsub\omega$ and $\Phi_k|_{\de B_r(x)}$ also converges in $C^0$ to $\Phi_\infty|_{\de B_r(x)}$ (up to subsequences).
		Let $\tilde\omega:=\omega\setminus\bigcup_{x\in\omega\cap\mathcal{A}}\bar B_r(x)$.
		
		Repeating the proof of \cref{limit.is.stat}
		with vector fields in $\fbvf$ supported outside $\Phi_\infty(\de\tilde\omega)$, we deduce that the varifold limit of $\Phi_k|_{\tilde\omega}$ is free boundary stationary outside this set;
		by \cref{var.conv}, this limit is $\vfd_{\tilde\omega}$.
		Since the images $\Phi_\infty(\de B_r(x))$, for $x\in\omega\cap\mathcal{A}$, have arbitrarily small diameter (see, e.g., \cite[Lemma~A.3]{pigriv}), we deduce that
		$\vfd_\omega$ is free boundary stationary outside ($\Phi_\infty(\de\omega)$ and) a finite set $F$. However, since $\Phi_k$ also converges to the free boundary stationary varifold $\vfd_\infty\ge\vfd_\omega$,
		by \cref{mono.easy.claim} we get $|\vfd_\omega|(B_s(p))\le Cs^2$ for $p\in F$.
		Hence, given $X\in\fbvf$ supported outside $\Phi_\infty(\de\omega)$, we can multiply it by the product $\Pi_{p\in F}\varphi_p$ of cut-off functions $\varphi_p$, with $\varphi_p=0$ on $B_{s/2}(p)$, $\varphi_p=1$ outside $B_s(p)$ and $|d\varphi_p|\le Cs^{-1}$. It is then straightforward to check that the stationarity with respect to the cut-off vector field gives the one for $X$, as $s\to 0$.
		The proof that $\vfd_\omega$ is stationary for a.e.\ $\omega\cptsub\operatorname{int}(\Sigma)$ is analogous.
		
		Finally, we show how to obtain a weakly conformal reparametrization. Note that, by \cref{struct.n},
		\begin{align*}
			&N|\de_1\Phi_\infty\wedge\de_2\Phi_\infty|\ge\mz|d\Phi_\infty|^2
		\end{align*}
		a.e.\ in a local conformal chart $h:U\to U'$ (with $U\subseteq\Sigma$), since the left-hand side is the density of $\nu_\infty$ (in $U'$) and thus, for any open set $V\cptsub U\cap\tilde\Sigma$,
		\begin{align*}
			&\int_V\mz|d\Phi_\infty|^2\,d\mathcal{L}^2
			\le\liminf_{k\to\infty}\nu_k(V)\le\nu_\infty(\bar V)
			=\int_{\bar V}N|\de_1\Phi_\infty\wedge\de_2\Phi_\infty|\,d\mathcal{L}^2,
		\end{align*}
		from which the previous claim follows. Call $\bar C$ an upper bound for $N$
		and assume that the image $U'$ of the chart is either the ball or the half-ball. Arguing as in the first part of the proof of \cite[Theorem~4.1]{pigriv}, we can find a $\frac{\bar C^2-1}{\bar C^2+1}$-quasiconformal homeomorphism $\psi:\C\to\C$
		such that $\Phi_\infty\circ h^{-1}\circ\psi^{-1}$ is weakly conformal on $\psi(U')$;
		using the Riemann mapping theorem and Carath\'eodory's theorem, by composing $\psi$ with a conformal map, we can replace $\psi$ with a homeomorphism $U'\to U'$, with the additional property that it preserves $U'\cap\{\Im(z)=0\}$ (as a set) in the half-ball case.
		Recall that $\psi^{-1}$ is quasiconformal, as well (see, e.g., \cite[Theorem~4.10 and Proposition~4.2]{imayoshi}).
		
		Set $\theta:=h^{-1}\circ\psi^{-1}:U'\to U$. Note that, given two overlapping charts $U_1,U_2$,
		the corresponding homeomorphisms $\theta_1$ and $\theta_2$ differ by a conformal map, namely
		$\theta_1^{-1}\circ\theta_2$ is conformal on $\theta_2^{-1}(U_1\cap U_2)$. This holds since a.e.\ the differential $d\Phi_\infty$ either vanishes or has rank $2$ and, by construction, $\theta_i$ is weakly conformal at a.e.\ $x_i$ such that $d\Phi_\infty(\theta_i(x_i))=0$; on the other hand, the two maps
		$d(\Psi_\infty\circ\theta_i)(x_i)=d\Psi_\infty(\theta_i(x_i))\circ d\theta_i(x_i)$,
		with $x_1:=\theta_1^{-1}\circ\theta_2(x_2)$, are both nontrivial linear conformal maps for a.e.\ $x_2$ such that $d\Phi_\infty(\theta_2(x_2))\neq 0$, so that $d\theta_1(x_1)^{-1}\circ d\theta_2(x_2)$ is conformal at these points.
		
		In this argument we used the facts that a quasiconformal homeomorphism carries negligible sets to negligible sets \cite[Lemma~4.12]{imayoshi} and satisfies the chain rule \cite[Lemma~III.6.4]{lehto}.
		
		Note that the Cauchy--Riemann equations satisfied by $\theta_1^{-1}\circ\theta_2$ give its smoothness away from the boundary and, by the Schwarz reflection principle, this map is smooth up to $\theta_2^{-1}(\de\Sigma\cap U_1\cap U_2)$. Thus, the maps $\theta^{-1}$ define an atlas for a new smooth and conformal structure on $\Sigma$;
		calling $\Sigma'$ a copy of $\Sigma$ with this structure, we can just take $\varphi$ to be the identity $\Sigma'\to\Sigma$.
		
		Finally, as explained in \cref{full.stat} (whose proof does not use that $\Phi$ is weakly conformal), the stationarity property holds on $\Sigma$ for all domains; the same then holds for $\Sigma'$.
	\end{proof}

	\section{Degeneration of the conformal structure and bubbling}\label{deg.sec}
	
	In this section we describe how to recover all the area in the limit as a sum of masses of parametrized (free boundary) stationary varifolds, without the assumption that the maps $\Phi_k$ induce a fixed conformal structure on $\Sigma$.
	
	Namely, denoting $\vfd_k$ the varifold induced by $\Phi_k$ as in \cref{crit.sec}, we show that the limit varifold $\vfd_\infty$ is the superposition of finitely many parametrized free boundary stationary varifolds.
	
	Before dealing with possible concentration points, we focus on how to remove the assumption of the fixed conformal structure.
	
	First of all, recall that on the oriented surface $\Sigma$, which can be assumed to be connected, there exists a metric $g_k$ which is conformal to the induced metric $g_{\Phi_k}=\Phi_k^*g$, has constant Gaussian curvature either $1$, $0$, or $-1$, and makes the boundary $\de\Sigma$ geodesic.
	Precisely, the curvature is $1$ if $\Sigma$ is (diffeomorphic to) a sphere or a disk,
	$0$ if $\Sigma$ is a torus or an annulus, and $-1$ otherwise.
	This is in agreement with Gauss--Bonnet, which says that the sign of this constant curvature agrees with the sign of the Euler characteristic of $\Sigma$, given by $\chi(\Sigma)=2-2g-b$, with $g$ the genus and $b$ the number of boundary components.
	
	We also recall that $(\Sigma,g_k)$, up to a change of orientation, is conformal to a surface $\Sigma_k$ which is the standard sphere, hemisphere, a torus $\C/(\Z+\Z\lambda_k)$ (where we can assume $|\lambda_k|\ge 1$, $|\Re(\lambda_k)|\le\mz$), an annulus $S^1\times [0,\ell_k]$, when $\Sigma$ is (diffeomorphic to) the sphere, the disk, the torus, the annulus, respectively.\footnote{This is an easy consequence of the Riemannian uniformization theorem, applied to $(\Sigma,g_k)$ if $\de\Sigma=\emptyset$, or to the doubled surface obtained by gluing two copies of $\Sigma$ along $\de\Sigma$.}
	
	Hence, when $\Sigma$ is the sphere, up to precomposing $\Phi_k$ with a diffeomorphism
	we can assume that $\Phi_k$ induces the standard conformal structure on $\Sigma=S^2$;
	note that this leaves the diffeomorphism invariant conditions \cref{almost.crit} and \cref{entropy} unchanged. The same holds for the disk case.
	
	Before discussing the other situations, let us state a useful modification of \cref{quant.phi}.
	
	\begin{proposition}\label{quant.phi.mod}
		Consider open domains $U_k\subseteq \Sigma$ whose boundary $\de U_k$ is contained in the union of two compact curves $\alpha_{k,1}$ and $\alpha_{k,2}$, and call $d_{k,i}$ the diameter of $\Phi_k(\alpha_{k,i})$.
		Assume that $U_k$ does not contain any entire boundary component of $\Sigma$.
		Then
		\begin{align*}
			&\limsup_{k\to\infty}\nu_k(U_k)\le\delta\Big(\limsup_{k\to\infty}\max\{d_{k,1},d_{k,2}\},C\Big),
		\end{align*}
		unless the left-hand side is at least $c_Q$, the same constant appearing in \cref{quant.phi}.
		In the last inequality, $C$ is a constant depending only on $(\Phi_k)$ and the function $\delta$ is given by \cref{quant.vfd.soft}.
	\end{proposition}
	
	In this statement, $U_k$ may contain points in $\de\Sigma$ and $\de U_k=\bar U_k\setminus U_k$ denotes its topological boundary in $\Sigma$.
	
	\begin{proof}
		Note that $\Phi_k(\alpha_{k,i})$ is contained in a ball $\bar B_{d_{k,i}}(p_{k,i})$.
		After extracting a subsequence realizing $\limsup_{k\to\infty}\nu_k(U_k)$,
		we can also assume that $p_{k,i}$ and $d_{k,i}$ converge to $p_{i}$ and $d_i$.
		
		The proof is now analogous to the one of \cref{quant.phi}: the maps $\Phi_k|_{U_k}$ induce varifolds whose (subsequential) limit is free boundary stationary on the complement of $\bar B_{d_1}(p_{1})\cup \bar B_{d_2}(p_{2})$, has mass at most $Cr^2$ on balls of radius $r$, and has density bounded below by a constant $c<1$ (the same as in that proof). The claim follows from \cref{quant.vfd.soft}.
	\end{proof}
	
	\subsection{Flat case}
	We now treat the torus case in detail, deferring the other cases to a later discussion.
	
	If $\Sigma_k=\C/(\Z+\Z\lambda_k)$, setting $\ell_k:=|\lambda_k|\ge 1$, we can also assume that $\ell_k\to\ell_\infty\in[1,\infty]$ up to a subsequence.
	If $\ell_\infty<\infty$, assuming $\lambda_k\to\lambda_\infty$ and defining $\Sigma_\infty:=\C/(\Z+\Z\lambda_\infty)$, we can find diffeomorphisms $\varphi_k:\Sigma_\infty\to\Sigma$ such that
	the pullback of the conformal structure $[g_{\Phi_k}]$ converges smoothly to the flat one.
	
	Since the area of $\Phi_k$ is bounded,
	the sequence $\Phi_k\circ\varphi_k$ is then bounded in $W^{1,2}(\Sigma_\infty)$ and we can extract a subsequence converging to a weak limit $\Phi_\infty$. Defining the area measure $\nu_k$ on $\Sigma_\infty$ as in the previous section, note that again their limit in the sense of Radon measures (up to subsequences) is also equal to the limit of $\mz|d\Phi_k|^2\,\operatorname{vol}_{\Sigma_\infty}$.
	
	All the proofs in \cref{asympt.sec} carry over, just replacing $\Phi_k$ with $\Phi_k\circ\varphi_k$ and $(\Sigma,g_0)$ with $(\Sigma_\infty,g_{\Sigma_\infty})$. Assume in the sequel $\ell_k\to\ell_\infty=\infty$.
	
	\begin{rmk}
		Actually, in the proof of \cref{struct.n} we used the conformality of the maps $\Phi_k$; since the proof was local in $\operatorname{int}(\Sigma)$, we can precompose $\Phi_k$ with a conformal map $h_k:B_1^2\to(\Sigma,g_k)$ which is a diffeomorphism with the image and converges smoothly to the inverse of a conformal chart for $\Sigma=\Sigma_\infty$.
		The statement for the sequence $(\Phi_k)$ then follows from its validity for the conformal maps $\Phi_k\circ h_k$.
	\end{rmk}
	
	Note that, since $|\Re(\lambda_k)|\le\mz$, we can use instead $S^1\times\ell_k S^1$ as a domain for $\Phi_k$,
	with the induced conformal structure becoming asymptotically the flat one.
	Given a big parameter $L$, we can subdivide the circle $\ell_k S^1$ into $N_k$ arcs $I_{k,1},\dots,I_{k,N_k}$ with $L\le |I_{k,j}|\le 2L$. Note that the boundedness of the area of $\Phi_k$ gives
	\begin{align*}
		&\int_{S^1\times\ell_k S^1}|d\Phi_k|^2\le C
	\end{align*}
	for some constant $C$ independent of $k$.
	
	Hence, for each $k$, there is only a bounded amount of indices $j$ such that $\mz\int_{S^1\times I_{k,j}}|d\Phi_k|^2\ge\frac{c_Q}{8}$, for the constant $c_Q$ from \cref{quant.phi.mod}. Up to subsequences, we can then find a nonempty collection of arcs $J_{k,1},\dots,J_{k,h}$
	which are unions of the previous intervals, in such a way that
	\begin{align*}
		&L<\lim_{k\to\infty}|J_{k,j}|<\infty,\quad\operatorname{dist}(J_{k,j},J_{k,j'})\to\infty\text{ for }j\neq j'
	\end{align*}
	and $\mz\int_{S^1\times I_{k,j}}|d\Phi_k|^2<\frac{c_Q}{8}$ whenever $I_{k,j}$ is not included in one of the arcs $J_{k,1},\dots,J_{k,h}$. We now claim that
	\begin{align}\label{deg.torus.claim}
		&\limsup_{k\to\infty}\int_{S^1\times(\ell_k S^1\setminus\bigcup_{j=1}^h RJ_{k,j})}|d\Phi_k|^2\to 0\quad\text{as }R\to\infty,
	\end{align}
	provided $L$ was chosen big enough. Here $RJ_{k,j}\subseteq\ell_k S^1$ is the arc dilated by a factor $R$, with the same center.
	
	Once this is proved, we can fix $j\in\{1,\dots,h\}$ and, shifting $J_{k,j}$ to be centered at $0$, we obtain
	a (local) weak limit $\Phi_{\infty,j}:S^1\times\R\to\subman$ of the maps $\Phi_k$, viewing these as maps defined on bigger and bigger subsets of $S^1\times\R$. We can again repeat the analysis which was done in the previous section.
	
	Note that in the limit we get a map with domain $S^1\times \R$. Since this cylinder is conformally the same as the sphere minus two points, we can see the domain as the sphere: note that replacing the cylinder with the sphere preserves stationarity, by the same argument used in the proof of \cref{par.stat} to remove the set of atoms.
	
	By \cref{deg.torus.claim}, the sum of the masses of the limit varifolds for $j=1,\dots,h$ is equal to the limit of the area of $\Phi_k$, up to the contribution of concentration points in the $h$ copies of $S^1\times\R$. We will discuss later how to recover the area which gets concentrated at these points.
	
	In order to prove \cref{deg.torus.claim}, fix $k$ and $j$, and let $I_{k,s},\dots,I_{k,s+t}$ be the intervals lying between two consecutive arcs $J_{k,j}$ and $J_{k,j+1}$ (with indices modulo $N_k$ and modulo $h$).
	We claim that eventually we cannot have $\sum_{i=2}^{t'}\int_{S^1\times I_{k,s+i}}\mz|d\Phi_k|^2\ge\frac{c_Q}{2}$
	for any $1<t'<t$. If $t'$ is the minimum such index, since the energy carried by each $S^1\times I_{k,s+i}$ is at most $\frac{c_Q}{8}$ we deduce that the sum is less than $\frac{5}{8}c_0$.
	
	Since $|I_{k,i}|\ge L$, we can select $a\in I_{k,s+1}$ and $b\in I_{k,s+t'+1}$ such that $\int_{S^1\times\{a,b\}}|d\Phi_k|\le CL^{-1/2}$;
	we can apply \cref{quant.phi.mod} with $U_k:=S^1\times[a,b]$ and deduce that eventually $\int_{S^1\times[a,b]}\mz|d\Phi_k|^2$ is either at least $\frac{7}{8}c_Q$ or at most $2\delta(CL^{-1/2},C)$. Since the first possibility cannot happen, we are in the second case.
	Hence, we get $\frac{c_Q}{2}\le 2\delta(CL^{-1/2},C)$, which is a contradiction for $L$ big enough, since
	$\delta(CL^{-1/2},C)\to 0$ as $L\to\infty$.
	
	But then we can repeat the argument selecting $a'$ in the part of $RJ_{k,j}\setminus(R/2)J_{k,j}$ following $J_{k,j}$
	and $b'$ in the part of $RJ_{k,j+1}\setminus(R/2)J_{k,j+1}$ preceding $J_{k,j+1}$, with
	$\int_{S^1\times\{a',b'\}}|d\Phi_k|\le CR^{-1/2}$. We already know that the area carried by the region $S^1\times[a',b']$ is eventually less than $\frac{c_Q}{2}$, so we deduce that it is bounded by $2\delta(CR^{-1/2},C)$, and \cref{deg.torus.claim} follows.
	
	In the case of the annulus, namely $\Sigma_k=S^1\times[0,\ell_k]$, up to subsequences either we are in the easy case that $\ell_k$ has a limit in $(0,\infty)$,
	or $\ell_k\to\infty$, or $\ell_k\to 0$.
	The second case can be dealt with in the same way as before, by subdividing the interval $[0,\ell_k]$ and making sure that $I_{k,1}\subseteq J_{k,1}$ and $I_{k,N_k}\subseteq J_{k,h}$.
	In this case, $J_{k,j}$ produces again infinite cylinders, or equivalently spheres, in the limit for $1<j<h$.
	On the other hand, $J_{k,1}$ and $J_{k,h}$ produce (possibly constant) limit maps
	whose domain is $S^1\times[0,\infty)$, which is conformally the disk minus the origin. We can thus view their domain as the full disk.
	
	In the last case $\ell_k\to 0$, we can replace $S^1\times[0,\ell_k]$ with the conformally equivalent surface $[0,1]\times \ell_k^{-1}S^1$.
	We then subdivide the circle and argue in the same way as before.
	In the limit we get maps with domain $[0,1]\times\R$, which is conformally a disk (minus two boundary points which can be ignored).
	
	\subsection{Hyperbolic case}
	Finally, we explain how to deal with the hyperbolic case $\chi(\Sigma)<0$. In this case there is no straightforward description of all the possible conformal classes of surfaces.
	In case $\Sigma$ has no boundary, by Bers' theorem we can decompose $(\Sigma,g_k)$ into hyperbolic pairs of pants, with lengths of their boundaries bounded above in terms of the topology of $\Sigma$: see \cite[Theorem~IV.3.7]{hummel} for a self-contained proof.
	We call $\{\beta_i\}$ the collection of closed geodesics, depending on $k$ but with fixed cardinality, which bound the pairs of pants.
	Up to subsequences, we can assume that the combinatorial configuration of the decomposition does not depend on $k$, with a consistent labeling for the curves $\beta_i$, and that the length of $\beta_i$ converges to a finite number as $k\to\infty$.
	
	Then we can apply \cite[Proposition~IV.5.1]{hummel} to the connected components of the surface $(\Sigma,g_k)$, cut open along those geodesics $\{\beta_i\}_{i\in I}$ whose length converges to $0$. We get a possibly disconnected limit surface $\Sigma_\infty$, which equals a closed Riemann surface minus finitely many points (two for each degenerating $\beta_i$), and diffeomorphisms $\psi_k:\Sigma_\infty\to\Sigma\setminus\bigcup_{i\in I}\beta_i$
	such that the pullback metric $\psi_k^*g_k$ converges locally to the metric of $\Sigma_\infty$.
	Then we repeat the analysis with the maps $\Phi_k\circ\psi_k$ and obtain a limit parametrized varifold, whose domain $\Sigma_\infty$ can be replaced with a (possibly disconnected) closed surface. Apart from concentration points, part of the area of $\Phi_k$ could be concentrating in collar neighborhoods of the geodesics $\beta_i$, for $i\in I$. These neighborhoods can be conformally identified with cylinders $S^1\times [0,L_{k,i}]$, with $L_{k,i}\to\infty$ as $k\to\infty$, and one can recover the missing part of the area as in the degenerating cylinder case; note that the pieces $S^1\times J_{k,1}$ and $S^1\times J_{k,h}$ from that analysis have to be discarded, since their contribution is already given by $\Sigma_\infty$, while all the other pieces produce varifolds parametrized by spheres.
	
	If $\de\Sigma\neq\emptyset$, let us call $\gamma_1,\dots,\gamma_b$ the boundary components of $\Sigma$.
	We cannot directly decompose $\Sigma$ into pairs of pants whose boundary curves have bounded length, since the length of some $\gamma_i$ with respect to $g_k$ could fail to stay bounded as $k\to\infty$.
	
	Instead, we first glue two copies of $(\Sigma,g_k)$ along the geodesic boundary $\de\Sigma=\bigcup_{i=1}^b\gamma_i$, obtaining a hyperbolic surface $\tilde\Sigma_k$. This surface comes equipped with a canonical involution $i_k$, which flips the two glued copies.
	
	For a decomposition for $\tilde\Sigma_k$ as in the closed case, we can assume that all the simple closed geodesics of length less than $2\sinh^{-1}(1)$ appear in the collection $\{\beta_i\}$: see \cite[Lemma~IV.4.1]{hummel} and the proof of Bers' theorem.
	
	The thin part $T_k:=\{x\in\tilde\Sigma_k:\operatorname{inj}(x)\le\lambda\}$ is invariant under $i_k$, since $i_k$ is an isometry.
	For $\lambda$ small enough, it consists of finitely many disjoint annuli containing a (simple closed) geodesic of length at most $2\lambda$, which is then in $\{\beta_i\}$: see the proof of \cite[Proposition~IV.4.2]{hummel}, which also shows that the curves $\beta_j$ with length bigger than $2\lambda$ are disjoint from $T_k$. Hence, choosing $\lambda$ small enough, we can assume $\beta_j\cap T_k=\emptyset$ for the indices $j\nin I$ corresponding to non-degenerating curves.
	
	The boundary of $T_k$ has a constant geodesic curvature $\kappa=\kappa(\lambda)$.
	Let $S_k:=\tilde\Sigma_k\setminus\operatorname{int}(T_k)$.
	Taking a limit $\tilde\Sigma_\infty$ as in the previous discussion, the proof of \cite[Proposition~IV.5.1]{hummel} shows that we can assume $\psi_k^{-1}(S_k)$ to be a constant domain $S_\infty$, whose complement is the union of finitely many cusps $\{C_j\}_{j\in J}$. Namely, each $\bar C_j$ is isometric to the quotient of $\{\Im(z)\ge\Lambda\}\subset\mathbb{H}$ by the standard parabolic isometry $z\mapsto z+1$, for some $\Lambda>0$ depending on $\lambda$.
	
	The maps $\psi_k^{-1}\circ i_k\circ\psi_k$ converge locally smoothly to an isometry $i_\infty:\tilde\Sigma_\infty\to\tilde\Sigma_\infty$, since $i_k$ is an isometry for $\tilde\Sigma_k$.
	The components of $\de T_k$ meeting $\de\Sigma\subset\tilde\Sigma_k$ are necessarily invariant sets for $i_k$, so that $\de\Sigma$ meets $\de T_k$ orthogonally on $\de\Sigma\cap\de T_k$. Also, we have a lower bound on the injectivity radius on $S_k$; this implies that a shortest path $\alpha$ joining a point in $S_k\cap\gamma_i$ to another curve $\gamma_{i'}$ has length bounded below by $\lambda$, since the geodesic $i_k\circ\alpha$ has the same endpoints; similarly, a shortest path between two close points in $S_k\cap\gamma_i$ must be $\gamma_i$ itself. Also, the length of a geodesic $\gamma_i$ intersecting $S_k$ cannot be smaller than $2\lambda$.
	These remarks imply that on $S_\infty$ the one-dimensional submanifold $\psi_k^{-1}(\de\Sigma)$ converges graphically to a limit $\Gamma_\infty\subseteq\{x\in S_\infty:i_\infty(x)=x\}$, which meets $\de S_\infty=\de T_\infty$ orthogonally.
	
	Thus, the domains $\psi_k^{-1}(\Sigma)$ converge graphically on $S_\infty$ to a domain $S_\infty'$ bounded by $\Gamma_\infty$.
	If $C$ is an $i_k$-invariant component of $T_k$, either $i_k$ interchanges the two circles in $\de C$ or it preserves them (as sets).
	In the former case, the core geodesic of $C$ appears in both collections $\{\gamma_i\}$ and $\{\beta_j\}$, and equals $\de\Sigma\cap C$.
	In the latter case, there are just two diametrically opposite fixed points of $i_k$ on each circle, so $\de\Sigma$ splits $C$ into two isometric pieces; we can thus assume that $\psi_k^{-1}(\Sigma\cap C)$ equals two half-cusps in this case.
	
	Hence, $T_\infty':=\psi_k^{-1}(\Sigma\cap T_k)$ is a constant union of cusps and half-cusps. The union $S_\infty'\cup T_\infty'$ is the desired limit surface, which is a compact Riemann surface $\Sigma_\infty$ minus finitely many points (in the interior or on the boundary).
	The area contribution which gets lost because of degenerating geodesics can be recovered as in the case of degenerating tori or annuli.
	
	Note that $\Sigma_\infty$ has at least $b(\Sigma)-|I|$ boundary components.
	Also, the Euler characteristic of its double is
	\begin{align*}
		2(2-2g(\Sigma_\infty)-b(\Sigma_\infty))&=2\chi(\Sigma_\infty)=\chi(\tilde\Sigma_j)+2|I|=2\chi(\Sigma)+2|I| \\
		&=2(2-2g(\Sigma)-b(\Sigma))+2|I|
	\end{align*}
	and we deduce $\chi(\Sigma_\infty)\ge\chi(\Sigma)$, $g(\Sigma_\infty)\le g(\Sigma)$.
	
	Note, however, that the number of boundary components could increase in principle: for instance, if $\Sigma$ has genus one and one boundary component, $(\Sigma,g_k)$ could degenerate conformally into an annulus.
	
	\subsection{Concentration points}
	
	We finally deal with concentration points for the area, or equivalently for the Dirichlet energy.
	The problem is local; since there can be only finitely many concentration points, we can deal with just a single one.
	Let $U'$ denote the ball or the half-ball. Up to precomposing the maps $\Phi_k$ with suitable diffeomorphisms $U'\to U\subset\Sigma$,
	we can assume that the induced conformal classes converge smoothly to the standard one, and that we have the tight convergence
	\begin{align*}
		&\nu_k':=\frac{1}{2}|d\Phi_k|^2\,\mathcal{L}^2\weakto m\,\mathcal{L}^2+\alpha\delta_0
	\end{align*}
	of measures on $U'$.
	Looking at a sufficiently small neighborhood of the concentration point, we can assume that $\int_{U'}m<\frac{c_Q}{2}$, while from \cref{ac} we have the lower bound $\alpha\ge c_Q$.
	Let $B_{r_k}^2(x_k)$ be a ball of minimal radius such that $\int_{B_{r_k}^2(x_k)\cap U'}\mz|d\Phi_k|^2\ge\alpha-\frac{c_Q}{2}$, so that the integral is exactly $\alpha-\frac{c_Q}{2}$ and necessarily $r_k\to 0$, $x_k\to 0$. It suffices to show that
	\begin{align}\label{bubb.claim}
		&\limsup_{k\to\infty}\nu_k'((B_{R^{-1}}^2(x_k)\setminus B_{Rr_k}^2(x_k))\cap U')\to 0\quad\text{as }R\to\infty.
	\end{align}
	Once this is done, we deduce that the area (or Dirichlet energy) measures of $\Psi_k:=\Phi_k(x_k+r_k\cdot)$ converge subsequentially to a measure $\nu$ (on the plane or a upper half-plane) of total mass $\alpha$. There could be further concentration points for this new sequence of maps, but their masses are at most $\alpha-\frac{c_Q}{2}$: this is obvious if there are at least two such points;
	if there is only one point $\bar x$ of mass bigger than $\alpha-\frac{c_Q}{2}$, then eventually
	\begin{align*}
		&\int_{B_{r_k/2}(x_k+r_k\bar x)\cap U'}\mz|d\Phi_k|^2=\int_{B_{1/2}^2(\bar x)\cap U_k'}\mz|d\Psi_k|^2>\alpha-\frac{c_Q}{2},
	\end{align*}
	where $U_k':=r_k^{-1}(U'-x_k)$, contradicting the minimality of $r_k$. Thus, this blow-up process has to be iterated only a finite amount of times.
	
	The proof of \cref{bubb.claim} is similar to the one of \cref{deg.torus.claim}.
	Select radii $R^{1/2}r_k<a<Rr_k$ and $R^{-1}<b<R^{-1/2}$ such that
	\begin{align*}
		&\int_{\de B_a^2(x_k)\cap U'}|d\Phi_k|^2\le \frac{2C}{a\log R},\quad\int_{\de B_b^2(x_k)\cap U'}|d\Phi_k|^2\le \frac{2C}{b\log R},
	\end{align*}
	where $C$ is an upper bound for $\int_{U'}|d\Phi_k|^2$;
	this can be done since the right-hand sides integrate to $C$ on the two intervals.
	Now the length of $\Phi_k|_{\de B_a^2\cap U'}$ and $\Phi_k|_{\de B_b^2\cap U'}$ is bounded by
	$\frac{C}{\sqrt{\log R}}$, for a different constant $C$. Since the area of $\Phi_k$ between the two radii is bounded by $\alpha+\int_{U'} m-(\alpha-\frac{c_Q}{2})+o(1)$, whose limit is less than $c_Q$, for $R$ big enough we can apply \cref{quant.phi.mod}
	and deduce that
	\begin{align*}
		&\limsup_{k\to\infty}\nu_k'((B_b^2(x_k)\setminus B_a^2(x_k))\cap U')\le\delta\Big(\frac{C}{\sqrt{\log R}},C\Big).
	\end{align*}
	Since $B_{R^{-1}}^2(x_k)\setminus B_{Rr_k}^2(x_k)\subseteq B_b^2(x_k)\setminus B_a^2(x_k)$, this proves \cref{bubb.claim}.
	
	The limit maps produced by concentration points have domains which are the plane or a half-plane, hence conformally the sphere or the disk (minus one point).
	
	
	
	\begin{proof}[Proof of \cref{main.thm}]
		Thanks to the arguments from this section and the previous one, we obtain disjoint domains $U_{k,1},\dots,U_{k,N}\subseteq\Sigma$ such that
		the varifold induced by $\Phi_k|_{U_{k,i}}$ converges to a parametrized free boundary stationary varifold,
		as $k\to\infty$, and $\int_{\Sigma\setminus\bigcup_i U_{k,i}}\operatorname{vol}_{\Phi_k}\to 0$.
		
		Since we can merge the domains of these parametrized varifolds into a (possibly disconnected)
		compact Riemann surface, the statement follows.
	\end{proof}
	
	\section{Regularity}\label{reg.sec}
	
	From the previous section we know that the limit varifold $\vfd_\infty$ is a parametrized free boundary stationary varifold
	$(\Sigma',\Phi,N')$, for some weakly conformal map $\Phi:\Sigma'\to\subman$ with $\Phi(\de\Sigma')\subseteq\bsubman$
	and $N'\in L^\infty(\Sigma',\{1,2,\dots\})$.
	This parametrized varifold gives rise, in local charts for $\Sigma'$, to a local parametrized stationary varifold, as defined in \cite[Definition~2.9]{pigriv} (see also \cite[Remark~2.3]{pigriv}).
	The main result of that work, namely \cite[Theorem~5.7]{pigriv}, tells us that $N'$ is locally constant and $\Phi$ is a branched minimal immersion, on (the interior of) the components of $\Sigma'$ where $\Phi$ is not (a.e.) constant.
	
	Hence, in order to study the regularity of $\Phi$, we can discard these trivial components and replace $N'$ with $1$, without affecting the stationarity property enjoyed by the parametrized varifold (recall \cref{par.def}).
	
	\begin{rmk}\label{multone.rmk}
		Actually, the main result of \cite{multone} still applies in this setting, so that
		we have $N=1$ automatically for $\Phi$ arising as a limit of the maps $\Phi_k$.
		Indeed, \cite[Theorem~3.2]{multone} (with $\tau_k^4$ in place of $\tau_k^2$) still holds for smooth maps $\Psi_k:\bar B_R^2\to\subman_{p_k,\ell_k}$,
		with $\subman_{p_k,\ell_k}:=\ell_k^{-1}(\subman-p_k)$ ($p_k\in\subman$), which are $\tau_k^5$-critical on the interior for $E_{\tau_k}$ (where the term $\tau_k\operatorname{length}(\Phi|_{\de\Sigma})$ can be now ignored): see also \cref{tight}. The other property needed for that paper is that if $\Psi$, with values into $\subman_{p,\ell}$, is $\tau^5$-critical for $E_\tau$, then $\lambda^{-1}(\Psi-q)$ is $(\tau/\lambda)^5$-critical for $E_{\tau/\lambda}$ (with the manifold $\subman_{p+\ell q,\lambda\ell}$), whenever $\lambda\le 1$:
		this also holds and was obtained along the proof of \cref{struct.n}.
		
		Under the assumptions of \cref{asympt.sec}, this gives $\mz|d\Phi_j|^2\,\mathcal{L}^2
		\weakto|\de_1\Phi_\infty\wedge\de_2\Phi_\infty|\,\mathcal{L}^2
		\le\mz|d\Phi_\infty|^2\,\mathcal{L}^2$ in local charts, far from the concentration points.
		
		This implies that $\Phi_j\to\Phi_\infty$ in $W^{1,2}_{loc}$ here, so that $\Phi_\infty$ is still weakly conformal; hence, there was actually no need to reparametrize it.
		These remarks, however, are not needed in the present section, which establishes the regularity of general parametrized free boundary stationary varifolds.
	\end{rmk}
	
	For simplicity, since we will not need to refer back to the original setting, we will write $\Sigma$ in place of $\Sigma'$ in the rest of this section.
	In order to prove \cref{reg.thm.intro}, we wish to show the following result. The fact that $\Phi|_{\Sigma\setminus\de\Sigma}$ is a branched minimal immersion then follows as discussed in the last step of the proof of \cite[Theorem~5.7]{pigriv}.
	
	\begin{thm}\label{reg.thm.post}
		The map $\Phi:\Sigma\to\subman$ is $C^\infty$-smooth up to the boundary $\de\Sigma$ and has $\de_\nu\Phi\perp T\bsubman$ at $\de\Sigma$. 
	\end{thm}
	
	As already mentioned, the interior regularity was already established in \cite{pigriv}. Here we show again how it can be obtained when $N'=1$---a fact proved in \cite{riv.target} and used in \cite{pigriv}---presenting a slightly simplified proof which covers also the boundary regularity.
	
	We first show a simple strenghtening of \cref{par.stat}.
	In the sequel, given $\omega\subseteq\Sigma$ open, we let $\vfd_\omega:=\Phi_*(\omega)$.
	
	\begin{proposition}\label{full.stat}
		The map $\Phi$ is continuous and the stationarity (respectively, free boundary stationarity) of $\vfd_\omega$ in \cref{par.def}
		holds for any domain $\omega\cptsub\Sigma\setminus\de\Sigma$ (respectively, $\omega\subseteq\Sigma$).
	\end{proposition}
	
	\begin{proof}
		The continuity of $\Phi$ can be obtained by the same arguments used in the proof of \cref{ac}.
		
		As for the second statement, given $\omega\subseteq\Sigma$ and a vector field $X\in\fbvf$ supported outside $\Phi(\de\omega)$,
		we can find a nonnegative smooth function $\rho\in C^\infty_c(\omega)$ such that $\rho=1$ near the compact set $\Phi^{-1}(\operatorname{spt}(X))\cap\omega$.
		The stationarity of $\vfd_\omega$ against the vector field $X$ then follows from the same property for the varifolds $\vfd_{\{\rho>\lambda\}}$, for $0<\lambda<1$, each of which agrees with $\vfd_\omega$ near $\operatorname{spt}(X)$.
		The proof in the case $\omega\cptsub\Sigma\setminus\de\Sigma$ is analogous.
	\end{proof}
	
	Let us fix a metric on $\Sigma$, compatible with the conformal structure.
	As in \cite{riv.target}, we first show that
	\begin{align}\label{g.is.reg.zero}
		&\Phi\text{ is smooth near }\mathcal{G}',
	\end{align}
	with $\mathcal{G}'\subseteq\Sigma\setminus\de\Sigma$
	defined to be the set of points $x$ such that $d\Phi(x)$ has rank $2$ and, in a chart centered at $x$, $\int_{B_r^2}|d\Phi-d\Phi(0)|^2\,d\mathcal{L}^2=o(r^2)$.
	
	Before proving this, let us set $\mathcal{B}':=\Sigma\setminus\mathcal{G}'$, $\mathcal{B}:=\Phi^{-1}(\Phi(\mathcal{B}'))$ and $\mathcal{G}:=\Sigma\setminus\mathcal{B}$.
	
	\begin{rmk}\label{g.b.sat}
		Note that $\mathcal{B}$ and $\mathcal{G}\subseteq\mathcal{G'}$
		are both $\Phi$-saturated: this means that whenever $\Phi(x)=\Phi(y)$ and $x\in\mathcal{B}$, the same holds for $y$, and similarly for $\mathcal{G}$.
		
		Arguing as in the proof of \cref{ac}, we have
		\begin{align*}
			&|\vfd_\Sigma|(\Phi(\mathcal{B}'))\le C\int_{\mathcal{B}'}|d\Phi|^2\,\operatorname{vol}_\Sigma=0.
		\end{align*}
		Hence, as $|\vfd_\Sigma|=\Phi_*(\mz |d\Phi|^2\,\operatorname{vol}_\Sigma)$, we get $d\Phi=0$ a.e.\ on $\mathcal{B}$.
	\end{rmk}
	
	\begin{proof}[Proof of \cref{g.is.reg.zero}]
	Given $x\in\mathcal{G}'$, we can choose a conformal chart centered at $x$, mapping a neighborhood $U$ of $x$ to $B_1^2$.
	Viewing $\subman\subset\R^\envdim$, we can then select an arbitrarily small radius $r>0$ such that $\Phi(ry)=\Phi(0)+d\Phi(0)[ry]+o(r)$, for $|y|=1$ (see, e.g., \cite[Lemma~A.4]{pigriv}).
	
	Moreover, $\int_{B_r^2}\mz|d\Phi|^2\,d\mathcal{L}^2=\pi s^2+o(r^2)$, with $s:=|\de_1\Phi|(0)r=|\de_2\Phi|(0)r$. Hence, assuming that the above error $o(r)$ is less that $\delta r$, for a fixed $\delta$ small enough,
	we can apply Allard's regularity result \cite[p.~466]{allard} (see also \cite[Theorem~23.1]{simon})
	on the ball $B_{(1-\delta)s}^\envdim(\Phi(0))$,
	where the varifold $\vfd_{B_r^2}$ has generalized mean curvature bounded in $L^\infty$, small excess (for $r$ small), and total mass $\pi(1-\delta)^2 s^2+O(\delta)s^2$.
	
	We deduce that on some ball $B_\theta^\envdim(\Phi(0))$ the varifold $\vfd_{B_r^2}$ agrees with the graph $S$ of a smooth function $f:\R^2\to\R^{\envdim-2}$, with multiplicity one, up to rotating the coordinates.\footnote{The smoothness of $f$ can be assumed by standard Schauder theory, since $f$ satisfies an elliptic equation on a small ball.}
	
	Selecting a new radius $r'$ much smaller than $\theta$,
	such that $\Phi(r'y)=\Phi(0)+d\Phi(0)[r'y]+o(r')$,
	from the continuity of $\Phi$ we deduce that $|\vfd_{B_{r'}^2}|$ is supported in $S$. Hence, viewing $\mathcal{G}\cap U$ as a subset of $B_1^2$ and setting $\tilde{\mathcal{G}}:=\mathcal{G}\cap\bar B_{r'}^2$,
	from $|\vfd_{B_{r'}^2}|=(\Phi|_{B_{r'}^2})_*(\mz|d\Phi|^2\,\mathcal{L}^2)$
	we deduce $\Phi(y)\in S$ for all $y\in\tilde{\mathcal{G}}$.
	
	Thus, the map $\operatorname{dist}(\Phi,S)$ is $W^{1,2}$ on $B_{r'}^2$ and vanishes on $\tilde{\mathcal{G}}$, and hence its differential vanishes a.e.\ here. But $d\Phi=0$ a.e.\ on $\mathcal{B}$; it follows that this function is constant, giving $\Phi(\bar B_{r'}^2)\subseteq S$. Thus, $\Phi|_{\bar B_{r'}^2}$ factors as $(\operatorname{id}\times f)\circ\Psi$ for a suitable map $\Psi\in C^0\cap W^{1,2}(\bar B_{r'}^2,\R^2)$.
	By the chain rule, any point $y\in\tilde{\mathcal{G}}$ is necessarily Lebesgue for $d\Psi$, with $d\Psi(y)$ invertible.
	
	For any $y\in\tilde{\mathcal{G}}$ there exist arbitrarily small radii $s$ such that $\vfd_{B_s^2(y)}$ is supported in $S$ and has density at least one at $\Phi(y)$.
	As $\vfd_{B_{r'}^2}$ has multiplicity one on $B_\theta^\envdim$, this implies that $\Phi$ is injective on $\tilde{\mathcal{G}}$.
	
	But then, recalling \cref{g.b.sat}, it follows that
	$\Phi(y)$ is disjoint from $\Phi(\bar B_{r'}^2\setminus\{y\})$ for all $y\in\tilde{\mathcal{G}}$,
	and the same follows for $\Psi$. Given $y\in\tilde{\mathcal{G}}$ close to $0$
	and choosing a homotopy in $\bar B_{r'}^2\setminus\{y\}$ between the circles $\de B_{r'}^2(0)$ and $\de B_s^2(y)$,
	with their canonical orientation, we deduce that the maps
	$\Psi|_{\de B_{r'}^2}-\Psi(y)$ and $\Psi|_{\de B_s^2(y)}-\Psi(y)$
	determine the same element in $\pi_1(\R^2\setminus\{0\})$.
	
	But the first map is homotopic to $\Psi|_{\de B_{r'}^2}-\Psi(0)$, provided $\Psi(y)$ is close enough to $\Psi(0)$, while the second is homotopic to $d\Phi(y)|_{S^1}$ if $s$ is selected in the same way as $r$. We deduce that $d\Psi$ is either always orientation preserving or always orientation reversing on $\tilde{\mathcal{G}}$, near $0$. Thus $\Phi$, in local coordinates for $S$,
	solves the Cauchy--Riemann equations (up to conjugation) near $0$, establishing \cref{g.is.reg.zero}.
	\end{proof}
	
	\begin{rmk}\label{g.open}
		We implicitly ask that the chosen representative of $d\Phi$ agrees with the classical differential on the regular set of $\Phi$.
		Hence, by what we just proved, $\mathcal{G}'$ is open.
		It follows that $\mathcal{B}'$ and $\mathcal{B}$ are closed,
		so that $\mathcal{G}$ is open again.
	\end{rmk}
	
	\begin{rmk}\label{g.is.regular}
		Given $x\in\mathcal{G}$ and a neighborhood $U\cptsub\operatorname{int}(\Sigma)$ such that $\Phi|_{\bar U}$ is a diffeomorphism with the image, we can express any section $w\in C^\infty_c(U)$ of $\Phi^*T\subman$ as $w=X(\Phi)$, where $X$ is a (smooth) vector field on $\subman$ vanishing near $\Phi(\de U)$. Hence, using \cref{full.stat}, we get
		\begin{align*}
			&\int_U\ang{\nabla w,d\Phi}\,\operatorname{vol}_\Sigma=0,
		\end{align*}
		so that $\Phi$ solves the harmonic map equation $\nabla^*d\Phi=0$ on $\mathcal{G}$.
	\end{rmk}
	
	In order to show \cref{reg.thm.post}, let $y\in\Sigma$ and pick a conformal chart $U\to U'$ centered at $y$,
	with image equal to $B_1^2$ if $y\nin\de\Sigma$ and to $B_1^2\cap\{\Im(z)\ge 0\}$ otherwise.
	By continuity of $\Phi$ we can assume that
	$\bar{\Phi(U')}$ is contained in a coordinate chart for $\subman$. We call $\{x^1,\dots,x^m\}$ the coordinates and we let $\Phi^i:=x^i\circ\Phi$. We can also require that $\bsubman$ corresponds to $\{x^{n+1}=\dots=x^m=0\}$
	if $y\in\de\Sigma$, with $g_{ij}=0$ for $i\le k$ and $j>k$ on this set.
	
	Then, writing $e_k:=\frac{\de}{\de x^k}$, it suffices to show that
	\begin{align}\label{reg.claim}
		&\int_{U'}\ang{\nabla(fe_k),d\Phi}\,d\mathcal{L}^2=0
	\end{align}
	for all $k=1,\dots,m$ and all nonnegative $f\in C^\infty_c(U')$, with the additional constraint
	$f\in C^\infty_c(U'\setminus\de U')$ if $k>n$ and $y\in\de\Sigma$, where we write $\de U':=U'\cap\{\Im(z)=0\}$.
	
	Indeed, once this is done, if $y\nin\de\Sigma$ then $\Phi=(\Phi^1,\dots,\Phi^m)$ is a weak solution of the system
	\begin{align*}
		&-\de_i(g_{jk}(\Phi)\de_i\Phi^j)+\Gamma_{pk}^j(\Phi) g_{jq}(\Phi)\de_i\Phi^p\de_i\Phi^q,
	\end{align*}
	where $\Gamma_{pk}^j$ is defined by the relation
	$\nabla_{e_p}e_k=\Gamma_{pk}^j e_j$.
	The smoothness of $\Phi$ then follows from \cref{app} and \cref{app.bis}.
	
	If instead $y\in\de\Sigma$, we get a weak solution to the system
	\begin{equation*}
		\left\{\begin{aligned}
		&-\de_i(g_{jk}(\Phi)\de_i\Phi^j)+\Gamma_{pk}^j(\Phi) g_{jq}(\Phi)\de_i\Phi^p\de_i\Phi^q=0, \\
		&\de_\nu\Phi_k=0\quad\text{on $\de U'$, for }k\le n, \\
		&\Phi_k=0\quad\text{on $\de U'$, for }k>n,
		\end{aligned}\right.
	\end{equation*}
	in the sense specified in \cref{weak.interp},
	and regularity follows again from \cref{app} and \cref{app.bis}.
	
	By the coarea formula, \cref{reg.claim} is equivalent to
	\begin{align*}
		&\int_0^\infty\bigg(-\int_{\de\{f>\lambda\}}\ang{e_k(\Phi),\de_\nu\Phi}+\int_{\{f>\lambda\}}\ang{\nabla(e_k\circ\Phi),d\Phi}\bigg)\,d\lambda=0.
	\end{align*}
	In order to conclude, we will show that the quantity between brackets vanishes for a.e.\ $\lambda$.

	\begin{proposition}
		For almost every value of $\lambda>0$, for $\omega:=\{f>\lambda\}\cptsub U'$ it holds
		\begin{align*}
			&-\int_{\de\omega}\ang{e_k(\Phi),\de_\nu\Phi}+\int_\omega\ang{\nabla(e_k\circ\Phi),d\Phi}=0.
		\end{align*}
	\end{proposition}
	
	\begin{proof}
		Fix $\lambda$ such that $\omega$ has smooth boundary, transverse to $\de U'$ if $y\in\de\Sigma$,
		and such that the trace $\Phi|_{\de\omega}$ is $W^{1,2}$, with differential given by the restriction of $d\Phi$ and vanishing a.e.\ on $\de\omega\cap\mathcal{B}$.
		For all $\epsilon>0$, we call $\mathcal{B}_\epsilon$ the closed $\epsilon$-neighborhood of $\mathcal{B}$ in $U'$.
	
		Take a smooth function $\rho$ vanishing near $\Phi(\de\omega\cap\mathcal{B}_\epsilon)$.
		Then $\Phi$ is a smooth immersion in a neighborhood of $S\cap\de\omega$, with $S:=\operatorname{spt}(\rho\circ\Phi)$,
		since $S\cap\de\omega\subseteq\mathcal{G}$.
		
		We can cover $S\cap\de\omega$ with finitely many disjoint closed arcs $\{\gamma_j\}\subseteq\mathcal{G}$, with endpoints in $\de U'\cup \mathcal{B}_\epsilon=\mathcal{B}_\epsilon$, so that $\Phi$ is an immersion near each of them.
		Fix now a smooth unit vector field $\tilde\nu$ on $\de\omega$ which points towards $\omega$, with $\tilde\nu\in T\de U'$ on the finite set $\de\omega\cap\de U'$. We can find functions $f_j:\gamma_j\to [0,1)$ such that the curves
		\begin{align*}
			&\tilde\gamma_j:=\{x+f_j(x)\tilde\nu(x)\mid x\in\gamma_j\}
		\end{align*}
		are disjoint, included in $\mathcal{G}$, have endpoints in $U'\setminus S$, and have images $\Gamma_j:=\Phi(\tilde\gamma_j)$ transverse to each other (meaning also self-transverse).
		Note that all $f_j$'s can be chosen arbitrarily close to $0$ in the $C^\infty$ topology.
		
		We now consider the domain
		\begin{align*}
			&\Omega:=\omega\setminus\bigcup_j\{x+sf_j(x)\tilde\nu(x)\mid 0\le s\le 1,\,x\in\gamma_j\}.
		\end{align*}
		Note also that we can assume the sets in the last union to be disjoint and
		\begin{align}\label{vanish.endpt}
			&\rho=0\text{ near }\Phi(\{x+sf_j(x)\tilde\nu(x)\mid 0\le s\le 1\})
		\end{align}
		whenever $x\in \mathcal{B}_\epsilon$ is an endpoint of one of the curves $\gamma_j$. This implies
		\begin{align}\label{relevant.boundary}
			&\de\Omega\cap S\subseteq\bigcup_j\operatorname{int}(\tilde\gamma_j),
		\end{align}
		where $\operatorname{int}(\tilde\gamma_j)$ denotes $\tilde\gamma_j$ minus the endpoints.
		
		Fix a smooth function $\chi:[0,\infty)\to[0,1]$ with $\chi=1$ on $[1,\infty)$ and $\chi=0$ on $[0,\mz]$.
		Let $\Gamma:=\bigcup_j\Gamma_j$ and $\chi_{\eta}:=\chi(\frac{\operatorname{dist}(\cdot,\Gamma)}{\eta})$.
		
		Let $F$ denote the closure of $\bigcup_j\Phi^{-1}(\Gamma_j)\setminus\bigcup_j\tilde\gamma_j$, together with all the endpoints of the curves $\tilde\gamma_j$. By transversality and conformality of $\Phi$, for each $x\in\bigcup_j\tilde\gamma_j\setminus F$ we have
		$\operatorname{dist}(\Phi(x-s\nu(x)),\Gamma)=s|\de_\nu\Phi(x)|+o(s)$, where $\nu$ is the outward unit normal for $\Omega$, and the gradient of $\operatorname{dist}(\Phi(\cdot),\Gamma)$ at $x-s\nu(x)$ is $-|\de_\nu\Phi(x)|\nu(x)+o(1)$, where $o(1)$ is infinitesimal as $s\to 0$ ($s>0$).
		These estimates hold uniformly on compact subsets of $\bigcup_j\tilde\gamma_j\setminus F$.
		
		Moreover, by transversality again, for any fixed small $r>0$ the support of $\chi_\eta\circ\Phi$ intersects the $r$-neighborhood $U_r$ of $\bigcup_j\tilde\gamma_j$ in the union of an $O(\eta)$-neighborhood of $\bigcup_j\tilde\gamma_j$, plus a set of measure $O(r\eta)$. In view of these remarks,
		\begin{align*}
			&\lim_{\eta\to 0}\int_{\Omega\cap U_r}\rho(\Phi)\ang{e_k(\Phi)\otimes d(\chi_\eta\circ\Phi),d\Phi} \\
			&=-\lim_{\eta\to 0} \sum_j\int_{\tilde\gamma_j}\int_0^1\chi'\Big(\frac{s|\de_\nu\Phi(x)|}{\eta}\Big)\frac{|\de_\nu\Phi(x)|}{\eta}\ang{(\rho e_k)(\Phi),\de_\nu\Phi}(x)\,ds\,dx+O(r) \\
			&=-\int_{\tilde\gamma_j}\ang{(\rho e_k)(\Phi),\de_\nu\Phi}+O(r).
		\end{align*}
		Also, note that $\Phi(\mathcal{B})\cap\Gamma=\emptyset$ by \cref{g.b.sat}; hence, for $\eta$ small, $\chi_\eta=1$ near $\Phi(\mathcal{B})$ and we deduce that $\operatorname{spt}((1-\chi_\eta)\circ\Phi)\subseteq\mathcal{G}$.
		Recalling also \cref{relevant.boundary}, we can integrate by parts as follows:
		\begin{align*}
			&\int_{\Omega\setminus U_r}\rho(\Phi)\ang{e_k(\Phi)\otimes d(\chi_\eta\circ\Phi),d\Phi} \\
			&=\int_{\Omega\setminus U_r}(1-\chi_\eta)(\Phi)\ang{e_k(\Phi)\otimes d(\rho\circ\Phi),d\Phi}
			+\int_{\Omega\setminus U_r}(\rho(1-\chi_\eta))(\Phi)\ang{\nabla(e_k(\Phi)),d\Phi} \\
			&\quad+\int_{\Omega\cap\de U_r}(\rho(1-\chi_\eta))(\Phi)\ang{e_k(\Phi),\de_\nu\Phi},
		\end{align*}
		where we used the harmonicity of $\Phi$ on $\mathcal{G}$.
		The convergence $(1-\chi_\eta)(\Phi)\to 0$ a.e.\ on $\Omega\setminus U_r$ and on $\de U_r$ (for $r$ small enough)
		implies that the right-hand side is infinitesimal as $\eta\to 0$.
		
		But, by the stationarity property of $\vfd_\Omega$, setting $X_\eta:=\rho\chi_\eta e_k$ we have
		\begin{align*}
			&\int_\Omega\ang{\nabla(X_\eta\circ\Phi),d\Phi}=0,
		\end{align*}
		since $X_\eta$ vanishes near $\Phi(\de\Omega)$ by the choice of $\chi_\eta$ and \cref{vanish.endpt}.
		Hence, from the previous computations we deduce
		\begin{align*}
			-\sum_j\int_{\tilde\gamma_j}\rho(\Phi)\ang{e_k(\Phi),\de_\nu\Phi}
			+\int_\Omega\ang{e_k(\Phi)\otimes d(\rho\circ\Phi),d\Phi}
			+\int_\Omega\rho(\Phi)\ang{\nabla e_k(\Phi)[d\Phi],d\Phi}=0.
		\end{align*}
		Letting $f_j\to 0$ we deduce our claim, provided we can replace $\rho$ with $1$.
		This is achieved as follows: the compact set $T:=\Phi(\de\omega\cap\mathcal{B}_\epsilon)$ has
		\begin{align*}
			&\mathcal{H}^1(T)\le\int_{\de\omega\cap\mathcal{B}_\epsilon}|d\Phi|.
		\end{align*}
		Hence, can cover $T$ with finitely many balls $B_{r_i}(p_i)$ intersecting $T$, such that
		\begin{align}\label{sum.radii}
			&2\sum_i r_i\le\mathcal{H}^1(T)+\epsilon
		\end{align}
		and $r_i<\epsilon$.
		Take now cut-off functions $0\le\rho_i\le 1$ which equal $0$ on $B_{r_i}(p_i)$ and $1$ on $\subman\setminus B_{2r_i}(p_i)$,
		with $|d\rho_i|\le Cr_i^{-1}$. Then the function $\rho:=\prod_i\rho_i$ satisfies
		\begin{align*}
			&\int_\omega|d\rho|(\Phi)|d\Phi|^2
			\le C\sum_i r_i^{-1}\int_{\omega\cap\Phi^{-1}(B_{2r_i}(p_i))}|d\Phi|^2
			\le Cr_i,
		\end{align*}
		because $(\Phi)_*(\mz|d\Phi|^2)\le\vfd_\Sigma$ and $\vfd_\Sigma(B_{2r_i}(p_i))\le Cr_i^2$ (see \cref{mono.easy.claim}).
		Note that the right-hand side of \cref{sum.radii} becomes infinitesimal as $\epsilon\to 0$, as $\int_{\de\omega\cap\mathcal{B}}|d\Phi|=0$.
		
		Finally, writing $T_\epsilon$ and $\rho_\epsilon$ in place of $T$ and $\rho$ to emphasize the dependence on $\epsilon$,
		we have $\rho_\epsilon(\Phi)\to 1$ pointwise on $\mathcal{G}$:
		indeed, since $T_\epsilon\to\Phi(\de\omega\cap\mathcal{B})$ in the Hausdorff topology,
		if $\rho_\epsilon(\Phi(x))$ does not converge to $1$ then $\Phi(x)\in\Phi(\de\omega\cap\mathcal{B})$
		and thus, by \cref{g.b.sat}, $x\in\mathcal{B}$. Hence,
		\begin{align*}
			0&=-\int_{\de\omega}\rho_\epsilon(\Phi)\ang{e_k(\Phi),\de_\nu\Phi}
			+\int_\omega\ang{e_k(\Phi)\,d(\rho_\epsilon\circ\Phi),d\Phi}
			+\int_\omega\rho_\epsilon(\Phi)\ang{\nabla e_k(\Phi)[d\Phi],d\Phi} \\
			&\to -\int_{\de\omega}\ang{e_k(\Phi),\de_\nu\Phi}
			+\int_\omega\ang{\nabla e_k(\Phi)[d\Phi],d\Phi},
		\end{align*}
		as desired.
	\end{proof}

	\appendix
	\section*{Appendix}\label{app.sec}
	\renewcommand{\thesection}{A}
	\setcounter{definition}{0}
	\setcounter{equation}{0}
	
	\begin{proposition}\label{app}
		A continuous, $W^{1,2}$ map $u:B_1^2\to\R^m$ solving a linear system of the form
		\begin{align*}
			&-\de_i(g_{jk}\de_i u^j)+b_{kpq}\de_i u^p\de_i u^q=0,
		\end{align*}
		with $g\ge\lambda>0$ symmetric and continuous and $b$ bounded,
		is $W^{1,r}_{loc}$ for all $r<\infty$.
		
		The same holds for $u$ defined on the half-ball $U':=B_1^2\cap\{\Im(z)\ge 0\}$,
		if in addition we have
		\begin{align*}
			&\de_\nu u^k=0\text{ for }k\le n,\quad u^k=0\text{ for }k>n,
		\end{align*}
		as well as $g_{ij}=0$ for $i\le n$, $j>n$, on the boundary $\de U'$, for some $0\le n\le m$.
	\end{proposition}
	
	\begin{rmk}\label{weak.interp}
		The condition $\de_\nu u^k=0$ could be written more faithfully as $g_{jk}\de_\nu\Phi^j=0$ and is of course meant in a weak sense, coupled with the equation: namely,
		we require $\int_{U'}(g_{jk}\de_i f\de_i u^j+b_{kpq}f\de_i u^p \de_i u^q)=0$
		for all $f\in C^\infty_c(U')$ and $k\le n$, allowing $f$ to be nonzero on $\de U'$.
	\end{rmk}
	
	\begin{proof}
		Assume $u$ is a solution on the unit ball. Then, for any ball $B_{2r}^2(x)\subseteq B_1^2$,
		we can integrate the equation against $\eta^2(u-(u)_{B_{2r}^2(x)})$, where $\eta\in C^\infty_c(B_{2r}^2(x))$ is a cut-off function satisfying $\eta=1$ on $B_r^2(x)$ and $|d\eta|\le\frac{2}{r}$. Recall that the notation $(u)_S$ indicates the average of $u$ on a set $S$. This gives
		\begin{align*}
			&\lambda\int\eta^2|du|^2
			\le C\int\eta|du|\,|d\eta|\,|u-(u)_{B_{2r}^2(x)}|
			+C\int\eta^2|du|^2\operatorname{osc}(u,B_{2r}^2(x))
		\end{align*}
		and, applying Young's inequality, it follows that
		\begin{align*}
			&\int_{B_r^2(x)}|du|^2\le Cr^{-2}\int_{B_{2r}^2(x)}|u-(u)_{B_{2r}^2(x)}|^2
			\le Cr^{-2}\Big(\int_{B_{2r}^2(x)}|du|\Big)^2
		\end{align*}
		whenever $\operatorname{osc}(u,B_{2r}^2(x))$ is small enough.
		The classical Gehring's lemma (see, e.g., \cite[Theorem~V.1.2]{giaquinta}) then implies
		that $du\in L^r(B)$ for some $r>2$ and any fixed ball $B\cptsub B_1^2$ (with $r$ depending on $B$). Then the nonlinear term $b_{kpq}\de_i u^p\de_i u^q$ is $L^{r/2}(B)$
		and standard elliptic regularity theory gives $du\in L^s_{loc}(B)$, with $\frac{1}{s}=\frac{2}{r}-\mz$, so that $s>r$; iterating, we get $du\in L^t_{loc}$ for any $t<\infty$.
		
		If we are in the half-ball case, then we can reduce to the previous case by reflection.
		We extend $g$ and $u$ to $\tilde g$ and $\tilde u$ on the ball $B_1^2$, by means of the formula
		\begin{align*}
			&g(s,-t):=Ug(s,t)U,\quad \begin{pmatrix}\tilde u^1 \\ \vdots \\ \tilde u^m\end{pmatrix}(s,-t):=U\begin{pmatrix}u^1 \\ \vdots \\ u^m\end{pmatrix}(s,t)
		\end{align*}
		for $(s,-t)$ in the lower half-ball, with $U:=\begin{pmatrix}I_n & \\ & -I_{m-n}\end{pmatrix}$. Note that, by our hypotheses on $g$, $\tilde g$ is still continuous.		
		Also, it is straightforward to check that $\tilde u$ solves
		\begin{align*}
			&-\de_i(\tilde g_{jk}\de_i \tilde u^j)+\tilde b_{kpq}\de_i \tilde u^p\de_i \tilde u^q=0,
		\end{align*}
		with $\tilde b_{kpq}$ extending $b_{kpq}$ according to the following rule:
		if $k\le n$ then $\tilde b_{kpq}(s,-t):=b_{kpq}(s,t)$ if $p$ and $q$ belong to the same set in the partition $\{\{1,\dots,n\},\{n+1,\dots,m\}\}$, and $\tilde b_{kpq}(s,-t):=-b_{kpq}(s,t)$ otherwise;
		if $k>n$ then the opposite holds.
		Then from the case of the full ball we deduce $d\tilde u\in L^t_{loc}$ for any $t<\infty$.
	\end{proof}
	
	\begin{rmk}\label{app.bis}
		If the coefficients are smooth functions of $u$, then $u$ is smooth.
		To check this, note that in the full ball case $u$ is $C^{0,\alpha}_{loc}$ for any $\alpha<1$.
		The same is then true for the coefficients $g_{jk}(u)$.
		Since the nonlinearity $b_{kpq}\de_i u^p\de_i u^q$ belongs to $L^r_{loc}$ for all $r<\infty$,
		classical Schauder theory then gives $du\in C^{0,\alpha}_{loc}$ for all $\alpha<1$
		and bootstrapping we reach $u\in C^\infty$.
		
		In the half-ball case, we can still argue in the same way that $d\tilde u\in C^{0,\alpha}_{loc}$ for all $\alpha<1$.
		So $\tilde g$ is locally Lipschitz and we deduce $\tilde u\in W^{2,r}_{loc}$ for all $r<\infty$.
		Differentiating the original equation in the first variable preserves the boundary conditions and leads to an equation of the form
		\begin{align*}
			&\de_i(g_{jk}\de_i(\de_1 u^j))+f_k=0
		\end{align*}
		with $f_k\in L^r_{loc}$ for all $r<\infty$, and the same reflection trick (applied to $w:=\de_1 u$)
		gives $\de_1 u\in W^{2,r}_{loc}$ for all $r<\infty$. Iterating we get the same for all derivatives $\de_1^k u$.
		Now the equation allows to deduce inductively that $u\in W^{k,r}_{loc}$ for all $k$,
		since $g_{jk}(u)\Delta u^j=-\de_i(g_{jk}(u))\de_i u^j+b_{kpq}(u)\de_i u^p\de_i u^q$;
		this expresses $\de_{22}u$ in terms of $\de_{11}u$ and lower order derivatives and
		hence, for any multi-index $\alpha=(\alpha_1,\alpha_2)$ with $\alpha_2\ge 2$,
		we deduce that $\de^\alpha u=\de_1^{\alpha_1}\de_2^{\alpha_2}u\in L^r_{loc}$ for all $r<\infty$ from the same property enjoyed by $\de_1^{\alpha_1+2}\de_2^{\alpha_2-2}u$ and lower order derivatives of $u$.
	\end{rmk}
	
	The following statements deal with general varifolds. It is clear that we can assume the smallness constant $c_V$ appearing in all of them to be always the same.
	
	\begin{lemmaen}\label{quant.vfd}
		There exists $c_V(\subman,\bsubman)>0$ with the following property.
		Given $p\in\bsubman$ and $0<s<c_V$, for any $2$-varifold $\vfd$ on $\mathcal{M}$ which is free boundary stationary outside $\bar B_s(p)$ and has density $\theta\ge\bar\theta$ on $\operatorname{spt}(|\vfd|)\setminus\bar B_s(p)$, either $\operatorname{spt}(|\vfd|)\subseteq B_{2s}(p)$ or
		$|\vfd|(\subman\setminus\bar B_s(p))\ge c_V\bar\theta$.
	\end{lemmaen}
	
	\begin{proof}
		Pick $\gamma>0$ small, to be fixed along the proof; we will choose $c_V\le\gamma$,
		so that the varifold is free boundary stationary outside $\bar B_\gamma(p)$
		Possibly multiplying $\vfd$ by $\bar\theta^{-1}$, we can assume $\bar\theta=1$.
		Note that if $q\in\operatorname{spt}(|\vfd|)\setminus B_{2\gamma}(p)$ then by \cref{mono.easy.claim2} we have
		\begin{align}\label{cv.spec}
		&|\vfd|(B_\gamma(q))
		\ge c(\subman,\bsubman)\gamma^2\theta(|\vfd|,q)
		\ge c(\subman,\bsubman)\gamma^2.
		\end{align}
		Otherwise, $|\vfd|$ is supported in $B_{2\gamma}(p)$.
		Assume we are in this second case and pick a set of coordinates $(x_1,\dots,x_m):B_{5\gamma}(p)\to\R^m$ centered at $p$,
		with $\bsubman$ corresponding to $\{x_{n+1}=\dots=x_m=0\}$. We can impose that $\|g_{ij}-\delta_{ij}\|_{C^1}\le\gamma$
		(in coordinates),
		for $\gamma$ small, independently of $p\in\bsubman$.
		
		On this ball, we define the vector field $X$ to be $X(x):=\chi(|x|)x_i\frac{\de}{\de x_i}$,
		where $\chi:[0,\infty)\to[0,1]$ is smooth and such that $\chi'\ge 0$ on $[0,3\gamma]$, $\chi=1$ on $[\frac{5}{3}s,3\gamma]$, $\chi=0$ on $[0,\frac{4}{3}s]\cup[4\gamma,\infty)$.
		Assuming $\{|x|\le 4\gamma\}\cptsub B_{5\gamma}(p)$, we can smoothly extend $X$ to all of $\subman$, with $X=0$ outside the ball.
		For $\gamma$ small enough (independently of $p$ and $s<\gamma$), the $C^1$ closeness of $g_{ij}$ to $\delta_{ij}$ guarantees
		\begin{align*}
		&\operatorname{div}_\Pi X\ge 0
		\end{align*}
		for all $(p,\Pi)\in\operatorname{Gr}_2(\subman)$ in the support of $\vfd$, since we can assume $\operatorname{spt}(|\vfd|)\subseteq\{|x|<3\gamma\}$:
		indeed, here the contribution of $\chi'$ is nonnegative, while the one of the position vector $x_i\frac{\de}{\de x_i}$
		is close to $2$ (multiplied by $\chi(|x|)$).
		Also, the inequality is strict if $|x(p)|\ge\frac{5}{3}s$. Moreover, $X$ is tangent to $\bsubman$. We can also assume that $\bar B_s(p)\cptsub\{|x|\le\frac{4}{3}s\}$; hence, we can test the stationarity of $\vfd$ against $X$ and reach the contradiction
		\begin{align*}
		&0=\int_{(p,\Pi)\in\operatorname{Gr}_2(\subman)}\operatorname{div}_\Pi X\,d\vfd(p,\Pi)>0
		\end{align*}
		unless $\operatorname{spt}(|\vfd|)$ is contained in $\{|x|\le \frac{5}{3}s\}$.
		Since the latter can be assumed to be included in $B_{2s}(p)$, the statement follows from \cref{cv.spec}.
	\end{proof}
	
	\begin{rmk}\label{quant.vfd.bis}
		The same statement holds if $\vfd$ is stationary, without the assumption $p\nin\bsubman$.
		The proof is analogous (but simpler, in that we do not need coordinates adapted to $\bsubman$).
	\end{rmk}

	\begin{lemmaen}\label{quant.vfd.soft}
		There exist $c_V>0$ and $\delta:(0,\infty)^2\to(0,\infty)$, with $\lim_{s\to 0}\delta(s,t)=0$ for every $t$, satisfying the following property.
		Given two points $p_1,p_2\in\subman$ and a radius $s>0$, let $B:=\bar B_s(p_1)\cup\bar B_s(p_2)$;
		if a $2$-varifold $\vfd$ on $\mathcal{M}$ is free boundary stationary outside $B$, has density $\theta\ge\bar\theta$ on $\operatorname{spt}(|\vfd|)\setminus B$ and satisfies the bound
		\begin{align*}
			&|\vfd|(B_r(q))\le c'r^2\quad\text{for all }q\in\subman,\ r>0,
		\end{align*}
		then either $|\vfd|(\subman)\le \bar\theta\delta(s,c'/\bar\theta)$ or $|\vfd|(\subman)\ge c_V\bar\theta$.
		The constant $c_V$ and the function $\delta$ depend only on $\subman$ and $\bsubman$.
	\end{lemmaen}

	\begin{proof}
		We can assume $\bar\theta=1$. From \cref{mono.easy.claim2} it follows that any nontrivial free boundary stationary varifold $\vfd'$ with density at least $1$ on $\operatorname{spt}(|\vfd|)$ has $|\vfd'|(\subman)\ge\lambda(\subman,\bsubman)$.
		Let $\delta(s,c')$ be the supremum of all possible masses $|\vfd|(\subman)$ which are smaller than $c_V$, for $\vfd$ as in the statement, with $c_V$ to be specified below. Take a sequence $s_k\to 0$ of positive numbers and a sequence $\vfd_k$ satisfying the assumptions with $s=s_k$, as well as $\delta(s_k,c')-2^{-k}<|\vfd_k|(\subman)<c_V$.
		
		Up to subsequences we get a limit varifold $\vfd_\infty$ which is free boundary stationary on the complement of two points $\bar p_1$ and $\bar p_2$. We still have $|\vfd_\infty|(B_r(q))\le c'r^2$ for all centers $q$ and all radii $r$. This upper bound implies easily that actually $\vfd_\infty$ is free boundary stationary on the full manifold: see the proof of \cref{par.stat} for the details. Also, by \cref{mono.easy.claim2} it has a lower bound $c\le 1$ for its density on $\operatorname{spt}(|\vfd_\infty|)$. Hence, $|\vfd_\infty|(\subman)\ge c\lambda$ unless $\vfd_\infty=0$.
		
		Since $|\vfd_\infty|(\subman)=\lim_{k\to\infty}|\vfd_k|(\subman)\le c_V$, choosing any $c_V<c\lambda$ forces $\vfd_\infty=0$, so that $\delta(s_k,c')\to 0$. This shows that $\delta(s,c')\to 0$ as $s\to 0$.		
	\end{proof}

\end{document}